\documentclass[11pt]{article}
\usepackage{amsmath,amsfonts,amssymb, amsthm,enumerate,graphicx}
\usepackage{tikz,authblk}
\usepackage{subfig}
\usepackage{url}
\usepackage{longtable}

\theoremstyle{plain}
\newtheorem{theo}{{\bf Theorem}}[section]
\newtheorem{lemma}[theo]{{\bf Lemma}}
\newtheorem{prop}[theo]{{\bf Proposition}}
\newtheorem{cor}[theo]{{\bf Corollary}}
\newtheorem{conj}[theo]{{\bf Conjecture}}
\theoremstyle{remark}
\newtheorem{defn}[theo]{{\bf Definition}}
\newtheorem{remark}[theo]{{\bf Remark}}

\newtheorem{eg}[theo]{{\bf Example}}

\newcommand{\TPSS}{S^{\hspace{.2mm}2} \mbox{$\times
\hspace{-2.8mm}_{-}$} \, S^{\hspace{.1mm}2}}

\author[1] {Biplab Basak}
\author[2] { Jonathan Spreer}

\affil[1] {Department of Mathematics, Indian Institute of Science,
Bangalore 560\,012, India. biplab10@math.iisc.ernet.in.}

\affil[2] {School of Mathematics and Physics, The University of Queensland,
Brisbane QLD 4072, Australia. j.spreer@uq.edu.au.}

\title{Simple crystallizations of 4-manifolds}

\vspace{-8mm}

\date{April 10, 2015}

\begin{document}
\maketitle

\vspace{-10mm}

\begin{abstract}
	Minimal crystallizations of simply connected PL $4$-manifolds
	are very natural objects. Many of their topological features
	are reflected in their combinatorial structure which, in addition,
	is preserved under the connected sum operation.

	We present a minimal crystallization of the standard PL K3 surface.
	In combination with known results this yields minimal crystallizations
	of all simply connected PL 
	$4$-manifolds of ``standard'' type, that is, all connected sums of 
	$\mathbb{CP}^2$, $S^2 \times S^2$, and the K3 surface. In particular,
	we obtain minimal crystallizations of a pair of homeomorphic 
	but non-PL-homeomorphic $4$-manifolds.

	In addition, we give an elementary proof that the minimal $8$-vertex 
	crystallization of $\mathbb{CP}^2$ is 
	unique and its associated pseudotriangulation is related to the $9$-vertex 
	combinatorial triangulation of $\mathbb{CP}^2$ by the minimum of four edge 
	contractions.
\end{abstract}

\noindent {\small {\em MSC 2010\,:} Primary 57Q15. Secondary 57Q05, 57N13, 05C15.

}

\medskip

\hrule

\bigskip

\section{Introduction}

In this article, we will consider the following straightforward generalization of simplicial complexes:
A $d$-dimensional simplicial cell complex, or $\Delta$-complex in the terminology of \cite{Hatcher2002AlgTop}, will be
called {\em $k$-simple}, $k \leq d$, if its set of $k$-dimensional faces forms a simplicial complex.

In particular, we are interested in $1$-simple simplicial cell complexes (which for brevity we will just
refer to as {\em simple}) where the underlying space
is a (simply connected) $4$-manifold and the $1$-skeleton equals the $1$-skeleton of a single $4$-simplex.
Such an object will be called a {\em simple contracted pseudotriangulation} of a $4$-manifold,
and can be described in terms of a $5$-colored graph,
which will be called a {\em simple crystallization} of the manifold. Converting between these two representations
is straightforward. Thus, statements about simple crystallizations will naturally transform to statements 
about simple contracted pseudotriangulations and vice versa. 

Simple contracted pseudomanifolds 
(and hence simple crystallizations) have a number of convenient properties, namely:

\newpage
\begin{enumerate}[(i)]
	\item they are simply connected by construction; 
	\item as a consequence, using Freedman's classification, the homeomorphism
		problem for simple contracted pseudomanifolds is decidable and can be 
		determined in a polynomial time procedure;
	\item for a given $4$-manifold $\mathbb{M}$ they have the minimum number of faces amongst all 
		pseudotriangulations of $\mathbb{M}$. In particular, their topologies are presented in very 
		compact forms;
	\item they always have five vertices and ten edges. Hence, their Euler characteristic,
		homology and the rank for their intersection form are
		determined by their number of facets. In particular, two simple pseudotriangulations 
		of two $4$-manifolds coincide in homology if and only if they have identical $f$-vectors;
	\item the connected sum of two simple crystallizations is again
		a simple crystallization. Thus, modulo the $11/8$-conjecture, simple crystallizations of all 
		topological types of simply connected PL $4$-manifolds can be constructed 
		from simple crystallizations of $S^4$, $\mathbb{CP}^2$, $S^2 \times S^2$ and the K3 surface,
		a prime (i.e., indecomposable by nontrivial connected sums) 
		simply connected PL $4$-manifold with even intersection form of signature $16$ and 
		rank $22$, first mentioned in \cite{Weil58K3Surf} (see also Section~\ref{intersection}).
\end{enumerate}
For these reasons, simple crystallizations, or simple contracted pseudotriangulations, 
are excellent objects to work with in the class of simply connected PL $4$-manifolds.
In many ways they correspond to the famous class of $3$-neighborly combinatorial $4$-manifolds in 
the world of simplicial complexes \cite{Casella01TrigK3MinNumVert,Kuehnel83The9VertComplProjPlane,Lutz03TrigMnfFewVertVertTrans}.
For instance, both simple crystallizations and $3$-neighborly combinatorial $4$-manifolds
are simply connected by construction, and for both settings the homology groups of the underlying manifold
are determined by their number of pieces.

However, not all simply connected PL $4$-manifolds of ``standard type'', i.e., connected sums of $\mathbb{CP}^2$,
$S^2 \times S^2$, the K3 surface and their copies with opposite orientation, can be described by $3$-neighborly simplicial complexes 
(for instance, $S^2 \times S^2$ does not admit such a triangulation, see \cite{Kuehnel83Uniq3Nb4MnfFewVert}).
Moreover, the essential $3$-neighborliness property is not preserved under taking the connected sum (cf. property (v) in the above list). 

\medskip
Simple crystallizations, and hence simple contracted pseudotriangulations, of simply connected $4$-manifolds
have been continuously studied over the last decades. The $8$-vertex crystallization of $\mathbb{CP}^2$ 
is a well-known object in the literature \cite{Cavicchioli08ClassComb4Mflds,Chiavacci93LinkingMinTrig,Gagliardo89CP2}.
Here, we complement this result by giving an elementary proof that there is a unique simple 
crystallization of the complex projective plane (this fact, of course, also follows from
the classification). 
A simple crystallization of $S^2 \times S^2$ is presented in \cite{Cavicchioli08ClassComb4Mflds,Ferri82S2xS2} and
all simple crystallizations with intersection form of rank up to two
have been recently classified due to work by Casali and Gagliardi 
\cite{Casali02CodeForMBiPartiteColouredGraphs}, Casali and Cristofori \cite{CasaliDUKEIII,Casali13ColouredGraphs,Casali14Cataloguing},
and Casali \cite{Casali12Crystallizations}.
In addition, most recently Casali {\em et al.} \cite{Casali14GemCompl} gave a Dehn-Sommerville type argument to link the gem-complexity
of simple crystallizations to the topology of the underlying manifold 
(cf. Remark~\ref{rem:DSE} in Section~\ref{sec:simpCryst}).

Furthermore, we present the first simple crystallization of the K3 surface.
This completes the list of simple crystallizations of all known simply connected prime PL $4$-manifolds.
As a consequence, by the connected sum property (v), we thus obtain simple crystallizations 
of simply connected manifolds of the form $k \,\mathbb{CP}^2 \,\, \#  \,\, l \,\overline{\mathbb{CP}^2}$ 
and $m \, K3 \,\, \# \,\, r \, S^2 \times S^2$, for all $k,l,m,r \geq 0$. This is a complete list
of all topological types of simply connected $4$-manifolds known to admit PL structures.

\medskip
Simple contracted pseudotriangulations have yet another important property. They are all isolated
global minima in the {\em Pachner graph}\footnote{The Pachner graph is the graph whose vertices
are complexes and two complexes are connected by an edge if and only if there is a
Pachner move turning one complex into the other. The Pachner graph is sometimes also referred to as the {\em flip graph}}
of pseudotriangulations (i.e., they don't allow $k$-moves, $k > 0$). By Pachner's theorem
\cite{Pachner87KonstrMethKombHomeo} (see for example \cite{Burton144Mflds} for a version for pseudotriangulations), 
the connected components of the Pachner graph precisely describe PL-homeomorphism classes. That is, two simplicial cell complexes
are PL-homeomorphic if and only if one can be turned into the other by a sequence of Pachner moves (in other words, if 
they are connected by a path in the Pachner graph). In practice, this result is a very useful tool 
to establish PL-equivalence between pairs of simplicial cell complexes
(see \cite{Bjoerner00SimplMnfBistellarFlips,simpcomp} for more about Pachner moves for simplicial complexes, 
and \cite{Burton11PachnerGraph,Burton09Regina,Burton13CombinatorialDiffeomorphisms} for the generalized triangulations setting).
Typically, this is done by repeatedly applying Pachner moves to both of the complexes. Each complex generated this way is known
to be PL-homeomorphic to the complex it has been constructed from. This way, we build two sets of PL-equivalent complexes, 
both representing connected subgraphs of the Pachner graph. Now the two complexes are PL-homeomorphic if and only if these two subgraphs
eventually overlap in a joint vertex in the Pachner graph (that is, if the two sets of PL-equivalent complexes overlap
in a joint complex). However, finding such a joint complex is extremely 
difficult due to the often very large number of complexes. Thus having well-defined, small regions in the Pachner 
graph were both subgraphs have a greater chance to meet is essential. 
Local minima are excellent candidates for such {\em meeting points}.

In Section~\ref{sec:compExpts}, we present a
heuristic routine to produce simple crystallizations from pseudotriangulations of $4$-manifolds.
Because of their special property, we believe that this heuristic method will be useful in a number of 
further applications, such as an ongoing project about PL-homeomorphisms for triangulated $4$-manifolds 
\cite{Burton13CombinatorialDiffeomorphisms,Burton144Mflds}.

\section{Preliminaries}

\subsection{Contracted pseudotriangulations}
\label{contrPT}

A $d$-dimensional CW-complex $K$ is said to be {\em regular} if the attaching maps which define the 
incidence structure of $K$ are homeomorphisms and the maximum dimension over all cells of $K$
is $d$. Given a regular CW-complex $K$, 
let ${\mathcal K}$ be the set of all closed cells of $K$ together with the empty set. 
Then ${\mathcal K}$ is a poset, where the partial ordering is the set inclusion. 
This poset ${\mathcal K}$ is said to be the {\em face poset} of $K$. Clearly, if 
$K$ and $L$ are two finite regular CW-complexes with isomorphic face posets then 
$K$ and $L$ are homeomorphic. Now, let $K$ be a regular CW-complex with partial ordering $\leq$
on its face poset $\mathcal{K}$. 
If $\beta\leq \alpha\in \mathcal{K}$ then we say $\beta$ is a {\em face} of $\alpha$. 
For $\alpha\in K$, the set 
$\partial \alpha := \{\gamma \in K \, : \, \alpha\neq \gamma \leq \alpha\}$ 
defines a subcomplex of $K$ with induced partial ordering and is called the {\em boundary} of $\alpha$. 
If all the maximal cells of a $d$-dimensional regular CW-complex $K$ are $d$-cells then $K$ is said 
to be {\em pure}. Maximal cells in a pure $CW$-complex $K$ are called 
{\em facets}, $0$-dimensional cells are called {\em vertices}, and $1$-dimensional
cells are called {\em edges} of $K$. More generally, the set of $i$-dimensional
faces of $K$ with its subfaces will be called the {\em $i$-skeleton} of $K$, denoted by $\operatorname{skel}_i (K)$. The vector $f(K) = (f_0 (K), f_1 (K), \ldots , f_d (K))$ will be called the {\em $f$-vector} of $K$ where $f_i(K)$ is the number of $i$ cells in $K$. The underlying topological space of $K$ is referred to as
the {\em geometric carrier} of $K$ which will be denoted by $| K |$.

\medskip
If all faces of a regular CW-complex $K$ are simplices then $K$
is often called a {\em generalized triangulation} (see for example 
\cite{Jaco03ZeroEffTriang}, where they are referred to as {\em triangulations}).
Generalized triangulations are predominantly used in $3$-manifold topology
and hyperbolic geometry where they are usually introduced as a set of tetrahedra 
together with face-pairings along their triangular faces. In particular, generalized
triangulations allow self-identifications of cells and often contain
no more than a single $0$-dimensional cell (hence, they are sometimes also called
{\em $1$-vertex triangulations}).

\medskip
Here, we want to focus on a slightly less general type of CW-complex.
A {\em simplicial cell complex} $K$ of dimension $d$ 
is a regular CW-complex such that the boundary of each face in $K$ is isomorphic 
(as a poset) to the boundary of a simplex of same dimension.
Note that every simplicial cell complex is a generalized triangulation
but the converse is not true. More precisely, in a simplicial cell complex no
self-identifications of faces are allowed. As a consequence, each simplicial
cell complex of dimension $d$ must have at least $d+1$ vertices.
If a $d$-dimensional simplicial cell complex $K$ has exactly $d+1$ vertices then 
$K$ is called {\em contracted}. 


If for a $d$-dimensional simplicial cell complex $K$ each $(d-1)$-face is contained in
exactly two facets of $K$, we say that $K$ is a {\em weak pseudomanifold}.
For $\alpha\in K$, the set $\{\sigma \in K \, : \, \alpha\leq \sigma\}$ is 
also a simplicial cell complex and is said to be the {\em star} of $\alpha$ in $K$, 
denoted by ${\rm st}_{K}(\alpha)$. Similarly, the set
$\{\sigma \setminus \alpha \in K \, : \, \alpha\leq \sigma\}$ is 
called the {\em link} of $\alpha$ in $K$, denoted by ${\rm lk}_{K}(\alpha)$. 
Here $\sigma \setminus \alpha$ denotes the set of all faces in $\sigma$
which are disjoint of $\alpha$. Furthermore, for any vertex $v$ of a $d$-dimensional 
simplicial cell complex $K$ the $(d-1)$-dimensional simplicial cell complex given by 
the boundary of a small neighborhood of $v$ in
$K$ (inside the interior of the subcomplex of all faces of $K$ containing $v$) is called the
{\em vertex figure} of $v$ in $K$ (note that in a simplicial complex $K$
a vertex figure of $v$ in $K$ is isomorphic to the link of $v$ in $K$). If all vertex figures
of $K$ are simplicial cell decompositions with their geometric carrier being PL-homeomorphic 
to the standard PL $(d-1)$-sphere then $K$ is said to be a {\em pseudotriangulation}
of a $d$-manifold. By construction, all pseudotriangulations of manifolds
are weak pseudomanifolds but the converse is not true.
Given a PL-manifold $\mathbb{M}$ we say that a pseudotriangulation $M$ is {\em (PL-)homeomorphic}
to $\mathbb{M}$ when $|M| \cong_{(PL)} \mathbb{M}$.

Now let $M$ be a $d$-dimensional weak pseudomanifold. Consider the graph 
$\Lambda(M)$ whose vertices are the facets of $M$ and the edges are pairs 
$(\{\sigma_1, \sigma_2\}, \gamma)$, where $\sigma_1$ and $\sigma_2$ are facets, and $\gamma$ 
is a common $(d-1)$-cell (i.e., $\gamma$ is a face of both $\sigma_1$ and $\sigma_2$). 
The graph $\Lambda(M)$ is said to be the {\em dual graph} or sometimes also the
{\em face pairing graph} of $M$. Observe that $\Lambda(M)$ of a weak pseudomanifold $M$ is 
a multi graph without loops. 

Pseudotriangulations of PL-manifolds are a straightforward generalization of 
combinatorial manifolds where the underlying CW-complex must be a simplicial complex 
(see \cite{Lutz11TrigMnflds,Spreer10Diss} for an introduction into combinatorial manifolds).
All together, we have three types of cell-decompositions, generalized triangulations,
simplicial cell complexes, and simplicial complexes which all are closely connected:
All simplicial cell complexes are generalized triangulations and the barycentric subdivision of
any generalized triangulation is a simplicial cell complex. All simplicial complexes in turn are simplicial
cell complexes and the barycentric subdivision of any simplicial cell complex is 
a simplicial complex \cite{Burton144Mflds}.

\medskip
An even richer set of classes of regular and simplicial CW-complexes of decreasing 
generality between simplicial cell complexes and simplicial complexes is given by the following.


\begin{defn}
	\label{defn:simple}
	Let $K$ be a $d$-dimensional simplicial cell complex. For $1 \leq k \leq d$, 
	$K$ is said to be {\em $k$-simple} if any set of $k+1$ vertices is in at most one $k$-cell.
\end{defn}

A $d$-dimensional simplicial cell complex $K$ is a simplicial complex if and only if $K$ is $d$-simple. 
If $K$ is contracted then $K$ is $k$-simple if any set of $k+1$ 
vertices is in a unique $k$-cell. If $K$ is $1$-simple we will call $K$ 
{\em simple}. From the definition we get the following.

\begin{lemma} \label{Lemma:simple-a}
	Let $M$ be a contracted $k$-simple pseudotriangulation of a closed connected $d$-manifold $\mathbb{M}$. 
	Then $k \leq d-1$, and $k=d-1$ if and only if 
	$M$ is a $2$-facet contracted pseudotriangulation of $S^{d}$.
\end{lemma}

\begin{lemma} \label{Lemma:simple-b}
	Let $M$ be a contracted pseudotriangulation of a closed connected $d$-manifold $\mathbb{M}$. 
	If $M$ is $k$-simple then $|M|$ is $k$-connected.
\end{lemma}

\begin{proof}
	For $k=0$ the statement directly follows since every contracted pseudotriangulation is connected.

	Let $k > 0$. $M$ is $d$-dimensional, $k$-simple ($k \leq d-1$), and contracted. Hence the full $k$-skeleton
	$\operatorname{skel}_k(M)$ of $M$ is contained in every $d$-simplex $\sigma \in M$.
	Thus $\operatorname{skel}_k(M) = \operatorname{skel}_k(\sigma) \subseteq \sigma$ 
	but on the other hand $\{ \sigma \} \subset M$. 
	Hence, we have $\pi_1(M, x) \leq \pi_1 (|\sigma|, x) =\{0\}$ for $x \in |\sigma|$ and 
	$H_i(M) \leq H_i (|\sigma|) = \{0\}$ for $1 \leq i \leq k$.
	It follows that $|M|$ is $k$-connected.
\end{proof}

\subsection{Colored Graphs}
\label{ssec:colGraphs}

In the following we will use the standard terminology for graphs as introduced
in \cite{Bondy08GraphTheory}.

\medskip
All graphs considered in this article are finite multi graphs without loops. 
Let $\Gamma = (V, E)$ be a graph and $U \subseteq V$ a subset of its vertices.
Then the {\em induced subgraph $\Gamma[U]$} is the subgraph of $\Gamma$
with vertex set $U$ containing all edges of $\Gamma$ with both endpoints 
lying in $U$. For $n\geq 2$, an $n$-cycle is a closed path with $n$ distinct 
vertices and $n$ edges. If vertices $a_i$ and $a_{i+1}$ are adjacent in an 
$n$-cycle for $1\leq i\leq n$ (addition is modulo $n$) then the $n$-cycle is 
denoted by $C_n(a_1, a_2, \dots, a_n)$. A graph $\Gamma$ is called 
{\em $h$-regular} or {\em $h$-valent} if the number of edges adjacent to 
each vertex is $h$. 

An {\em edge coloring} of  a graph $\Gamma = (V, E)$ is a surjective map $\gamma \colon E \to C$
such that $\gamma(e) \neq \gamma(f)$ whenever $e$ and $f$ are adjacent 
(i.e., $e$ and $f$ share a common vertex). The elements of the set 
$C$ are called the {\em colors}. If $C$ has $h$ elements then $(\Gamma, \gamma)$ 
is said to be an {\em $h$-colored graph}.

Let $(\Gamma,\gamma)$ be an $h$-colored graph with color set $C$. If $B \subseteq C$ 
with $k$ elements then the graph $(V(\Gamma), \gamma^{-1}(B))$ is a $k$-colored graph 
with coloring $\gamma|_{\gamma^{-1}(B)}$. This colored graph is denoted by $\Gamma_B$. 
Let $(\Gamma,\gamma)$ be an $h$-colored connected graph with color set $C$. If
$\Gamma_{C\setminus\{c\}}$ is connected for all $c\in C$ then  $(\Gamma,\gamma)$ 
is called {\em contracted}.

Let $\Gamma_1 = (V_1, E_1)$ and $\Gamma_2=(V_2, E_2)$ be two disjoint $h$-regular 
$h$-colored graphs with same color set $\{1, \dots, h\}$. Furthermore,
let $v_i \in V_i$ and let $u_{j,i} \in V_i$ be their neighbors in $\Gamma_i$ such that
the edge going from $v_i$ to $u_{j,i}$ is colored with color $j$, 
$1 \leq i \leq 2$, $1 \leq j \leq h$. Consider the graph $\Gamma$ obtained from 
$(\Gamma_1 \setminus\{v_1\}) \sqcup (\Gamma_2 \setminus \{v_2\})$ 
(here $\Gamma_i \setminus\{v_i\} = \Gamma_i[V_i\setminus\{v_i\}]$)
by adding $h$ new edges $e_j$ with colors $j$ respectively, $1 \leq j \leq h$, 
such that the end points of edge $e_j$ are $u_{j,1}$ and $u_{j,2}$.
The colored graph $\Gamma$ is called the {\em connected sum} of $\Gamma_1$ and $\Gamma_2$ and 
is denoted by $\Gamma_1\#_{v_1v_2}\Gamma_2$. Note that permuting the colors of $\Gamma_1$
gives rise to $h !$ ways to perform the connected sum with $\Gamma_2$ along 
$v_1$ and $v_2$.

\subsection{Crystallizations}
\label{cryst}

Crystallizations are colored graphs defining contracted pseudotriangulations. Hence,
they provide a way to visualize the essential properties of high dimensional manifolds
in a low-dimensional setting.

\medskip
Let $(\Gamma, \gamma)$ be a $(d+1)$-colored graph with color set 
$C = \{0, \dots,  d\}$, $d \geq 1$. Then a $d$-dimensional simplicial cell complex $M(\Gamma)$ 
can be defined as follows. For each $v\in V(\Gamma)$ we take a $d$-simplex $\sigma_v$ and label 
its vertices by $0, \dots, d$. If $u, v \in V(\Gamma)$ are joined by an edge $e$ and $\gamma(e) =  {i}$, 
then we identify the $(d-1)$-faces of $\sigma_u$ and $\sigma_v$ opposite to vertex ${i}$, such that
equally labeled vertices are identified. Since there is no identification within a $d$-simplex, 
$M(\Gamma)$ is a simplicial cell complex. We say that $(\Gamma, \gamma)$ 
{\em represents} the simplicial cell complex $M(\Gamma)$. 
Since, in addition, the number of $i$-labeled vertices of $M(\Gamma)$ is equal to the number of components of 
$\Gamma_{C\setminus\{{i}\}}$ for each $ {i}\in C$, the simplicial cell complex $M(\Gamma)$ is 
contracted if and only if $\Gamma$ is contracted \cite{Ferri86GraphTheoryCrystallizations}.

Hence, for a manifold $\mathbb{M}$ we will call a $(d+1)$-colored contracted graph 
$(\Gamma, \gamma)$ a {\em crystallization} of $\mathbb{M}$ if the simplicial cell complex $M(\Gamma)$ 
is a pseudotriangulation of $\mathbb{M}$.
Furthermore, the crystallization $(\Gamma, \gamma)$ of some closed $d$-manifold $\mathbb{M}$ either has two vertices 
(connected by $d+1$ edges, in which case $\mathbb{M}$ is $S^d$) or the number of edges between 
two vertices is at most $d-1$. We will call $(\Gamma, \gamma)$ $k$-simple
if $M(\Gamma)$ is $k$-simple. In \cite{Pezzana74Crystallizations}, Pezzana showed the following.

\begin{prop}[Pezzana] 
	\label{prop:pezzana74}
	Every connected closed PL-manifold admits a crystallization.
\end{prop}

Note that the analogous statement about simply connected manifolds and simple crystallizations 
would imply the Smooth Poincar\'e conjecture (cf. Section~\ref{ssec:s4}). 
However, since there exist simply connected topological 
$4$-manifolds with finite dimensional homology which admit an infinite number of PL structures this can not be
true in general.

\medskip
Crystallizations of manifolds admit a number of very useful combinatorial criteria which translate into
topological properties of the manifolds they describe. In the following we will list some
of these criteria.

\begin{prop}[Cavicchioli-Grasselli-Pezzana \cite{Cavicchioli80NormDecClsdNMflds}] \label{prop:ca-ga-pe80}
Let $(\Gamma,\gamma)$ be a crystallization of a $d$-manifold $\mathbb{M}$. Then $\mathbb{M}$ is orientable if and only if 
$\Gamma$ is bipartite.
\end{prop}

Let $(\Gamma, \gamma)$ be a $(d+1)$-colored graph with color set $C = \{0, \dots, d\}$.
For any  $k$-color set $D = \{i_1,i_2, \dots, i_k\} \subset C$, the number of components 
of the sub graph $\Gamma_{D}$ 
will be denoted by $g_{D}$ or sometimes just by $g_{i_1 i_2 \cdots i_k}$. 
With this setup in mind we can state

\begin{prop}[Gagliardi \cite{Gagliardo79CombCritCrystallizations}] \label{prop:gagliardi79a}
Let $(\Gamma,\gamma)$ be a contracted $4$-colored graph with $n$ vertices and color set $\{0,1,2,3\}$. 
Then $(\Gamma,\gamma)$ is a crystallization of a connected closed $3$-manifold if and only if

\begin{enumerate}[{\rm (i)}]
\item $g_{ij}=g_{kl}$ for $\{i,j,k,l\} =\{0,1,2,3\}$, and
\item $g_{01}+g_{02}+g_{03}=2+n/2$.
\end{enumerate}
\end{prop}

Let  $(\Gamma, \gamma)$ be a crystallization (with color set $C$) of a connected closed $d$-manifold $\mathbb{M}$.
Choose two colors $i,j \in C$, let $\{G_1, \dots,
G_{s+1}\}$ be the set of all connected components of $\Gamma_{C\setminus \{i,j\}}$, and $\{H_1, \dots, H_{t+1}\}$
be the set of all connected components of $\Gamma_{\{i, j\}}$. Since $\Gamma$ is regular, $H_k$ is an even
cycle for  $1 \leq k \leq t+1$ (note that in this case each $H_k$ is regular of degree two and $2$-colorable).
Note that, if $d=2$, then $\Gamma_{\{i, j\}}$ is connected and hence $H_1= \Gamma_{\{i, j\}}$. Take a set
$S = \{x_1, \dots, x_s, x_{s+1}\}$ of $s+1$ elements such that $x_m \in G_m$, $1 \leq m \leq s+1$. 
Choose a vertex $v_1$ in $H_k$ and let
$$H_k = C_{2l} (v_1 , v_2 , \ldots , v_{2l}),$$
where without loss of generality the edge between $v_1$ and $v_2$ has color $i$ and the edge between $v_2$ and $v_3$ has
color $j$. Define
\begin{align} \label{tildar}
\tilde{r}_k := x_{k_2}^{+1} x_{k_3}^{-1}x_{k_4}^{+1}  \cdots
x_{k_{2l}}^{+1}x_{k_1}^{-1}, \mbox{ for }1\leq k\leq t+1,
\end{align}
where $G_{k_h}$ is the component of $\Gamma_{C\setminus \{i,j\}}$ containing $v_h$.  For $1\leq k\leq t+1$, let
$r_k$ be the word obtained from $\tilde{r}_k$ by deleting $x_{s+1}^{\pm 1}$'s in $\tilde{r}_k$. 
Then we have

\begin{prop}[Gagliardi \cite{Gagliardi79FundGrpClsdNMfld}] \label{prop:gagliardi79b}
For $d\geq 2$, let  $(\Gamma, \gamma)$ be a crystallization of a connected closed $d$-manifold $\mathbb{M}$. For two
colors $i, j$, let $s$, $t$, $x_p$, $r_q$ be as above. If $\pi_1(\mathbb{M}, x)$ is the fundamental group of $\mathbb{M}$ at a
point $x$, then
$$
\pi_1(\mathbb{M}, x) \cong \left\{ \begin{array}{lcl}
\langle {x_1, x_2,\dots, x_s} ~ | ~ {r_1} \rangle & \mbox{if}
& d=2,   \\
\langle {x_1, x_2, \dots, x_s} ~ | ~ {r_1, \dots, r_t} \rangle
& \mbox{if} & d\geq 3.
\end{array}\right.
$$
\end{prop}
For more about presentations of fundamental groups of crystallizations see \cite{Basak13MinCryst}.

\subsection{4-manifolds and the intersection form} 
\label{intersection}

Given a closed topological $d$-manifold $\mathbb{M}$ we know that any smooth structure on $\mathbb{M}$ 
determines a PL structure on $\mathbb{M}$, and the converse holds for dimension $d\leq 6$, 
thus, for the remainder of this article we will regard PL structures as equivalent to smooth
structures and, since we are in the setting of triangulations, only refer to PL structures.

Given a closed oriented $4$-manifold $\mathbb{M}$, its {\em intersection form} is the symmetric $2$-form defined by
$$ Q_\mathbb{M}: H^2(\mathbb{M};\mathbb{Z}) \times H^2(\mathbb{M};\mathbb{Z})\to \mathbb{Z},\,\,Q_\mathbb{M}(\alpha,\beta)=(\alpha \smile \beta)[\mathbb{M}]$$ 
where $\smile$ denotes the cup-product.

$Q_\mathbb{M}$ is bilinear, symmetric and is presented by a quadratic matrix of size $\operatorname{rk} H^2(\mathbb{M};\mathbb{Z})$
of determinant $\pm 1$. The size is called the {\em rank} of $Q_\mathbb{M}$ and the difference between positive
and negative eigenvalues is referred to as its {\em signature}. If, for all $\alpha \in H^2 (\mathbb{M}; \mathbb{Z})$ we have that
$Q_\mathbb{M}(\alpha, \alpha)$ is an even number, then $Q_\mathbb{M}$ is called {\em even}. Otherwise, it is called {\em odd}. 
In order to define $Q_\mathbb{M}$ more geometrically, one can present classes $\alpha , \beta \in H^2(\mathbb{M};\mathbb{Z})$ 
by embedded surfaces $S_{\alpha}$ and $S_{\beta}$ of their Poincar\'e duals (this is always possible, see
\cite[Proposition 1.2.3]{Gompf}) and then equivalently define 
$Q_\mathbb{M}(\alpha,\beta)$ as the intersection number of $S_{\alpha}$ and $S_{\beta}$:
$$Q_\mathbb{M}(\alpha,\beta)= S_{\alpha} \cdot S_{\beta} .$$
Note that if $\mathbb{M}$ is simply connected, then $H_2(\mathbb{M};\mathbb{Z})$ 
is a free $\mathbb{Z}$-module and there are isomorphisms 
$H_2(\mathbb{M};\mathbb{Z}) \cong \mathbb{Z}^m$ where $m = b_2(\mathbb{M})$ 
(see \cite{Freedman90TopOf4Mflds, Gompf, Scorpan05WildWorldOf4Mflds} for more). 
From the definition we can deduce the following.
\begin{prop}
	\label{prop:connSum}
	Let $\mathbb{M}$ and $\mathbb{N}$ be oriented closed $4$-manifolds with intersection forms $Q_\mathbb{M}$ and $Q_\mathbb{N}$. 
	Then their connected sum $\mathbb{M} \# \mathbb{N}$ has intersection form $Q_\mathbb{M} \oplus Q_\mathbb{N}$.
\end{prop}

\bigskip
\begin{eg}
	\label{eg:intersection-form}
	Intersection forms of some well-known simply connected $4$-manifolds

\begin{enumerate}[(i)]
	\item The most common $4$-manifold $S^4$ does not have any $2$-homology. 
	Therefore we can take $\emptyset$ as its intersection form.
	\item The complex projective plane $\mathbb{CP}^2$ has intersection form 
	$Q_{\mathbb{CP}^2}=[+1]$ and the oppositely-oriented manifold
	$\overline{\mathbb{CP}^2}$ has intersection form 
	$Q_{\overline{\mathbb{CP}^2}}=[-1]$. Since $\mathbb{CP}^2 \cong \overline{\mathbb{CP}^2}$, 
	reversing orientation does not give a new manifold. 
	\item The manifold $S^2 \times S^2$ has intersection form 
	$$ Q_{S^2 \times S^2}= \begin{bmatrix} 0 & 1 \\  1 & 0 \end{bmatrix}.$$
	This matrix is often denoted by $H$ (from ``hyperbolic plane''). The oppositely-oriented manifold
	$\overline{S^2 \times S^2}$ has intersection form $Q_{\overline{S^2 \times S^2}}=-H$. Since $S^2 \times S^2 \cong 
	\overline{S^2 \times S^2}$, reversing orientation does not give a new manifold. 
	\item The twisted product $\TPSS$ has intersection form 
	$$Q_{\TPSS} = \begin{bmatrix} 1 & 1 \\  1 & 0 \end{bmatrix}.$$
	By a change of basis, we get $Q_{\TPSS}=[1] \oplus [-1]$. We will see that this proves
	$\TPSS \cong \mathbb{CP}^2 \# \overline{\mathbb{CP}^2}$. 
	\item The $E_8$-manifold $\mathbb{M}_{E_8}$ is a topological $4$-manifold with (even) intersection form
	 $$E_8 \,=\, Q_{\mathbb{M}_{E_8}}=\begin{bmatrix} 
	2 & 1 & 0 & 0 & 0 & 0 & 0 & 0 \\ 
	1 & 2 & 1 & 0 & 0 & 0 & 0 & 0 \\
	0 & 1 & 2 & 1 & 0 & 0 & 0 & 0 \\
	0 & 0 & 1 & 2 & 1 & 0 & 0 & 0 \\
	0 & 0 & 0 & 1 & 2 & 1 & 0 & 1 \\
	0 & 0 & 0 & 0 & 1 & 2 & 1 & 0 \\
	0 & 0 & 0 & 0 & 0 & 1 & 2 & 0 \\
	0 & 0 & 0 & 0 & 1 & 0 & 0 & 2 \\
	\end{bmatrix} .$$
	By Rohlin's theorem \cite{Rohlin84NewResults4Mflds} $\mathbb{M}_{E_8}$ 
	does not admit any PL structures, and thus is not
	of further interest when talking about simply connected PL $4$-manifolds. However, its intersection
	form $E_8$ will re-appear as a direct summand of the intersection form of the K3 surface.
	\item Recall that the K3 surface is a closed oriented connected and simply connected $4$-manifold. 
	Its intersection form is even of rank $22$ and signature $16$. In a suitable
	basis it is represented by the unimodular matrix $(-E_8) \oplus (-E_8) \oplus 3H$. 
	As a PL-manifold, it is prime (i.e., it can not be expressed as connected sum of PL-manifolds).
\end{enumerate}
\end{eg}
Freedman used the intersection for his celebrated classification of simply connected 
topological $4$-manifolds. More precisely, he proved the following statement.

\begin{theo}[Freedman \cite{Freedman82Top4DimMnf}]\label{Freedman1}
For every unimodular symmetric bilinear form $Q$ there exists 
a simply connected, closed, topological $4$-manifold $\mathbb{M}$ such 
that $Q_\mathbb{M} = Q$. If $Q$ is even, this manifold is unique (up to
homeomorphism). If $Q$ is odd, there are exactly two different 
homeomorphism types of manifolds with the given intersection form.
At most one of these homeomorphism types carries a PL structure.
Consequently, simply connected, PL $4$-manifolds are determined 
up to homeomorphism by their intersection forms.
\end{theo}
Using the classification theorem we now can state the converse of Proposition~\ref{prop:connSum}
for topological 4-manifolds.
\begin{cor}
	\label{Lemma:direct}
	If $\mathbb{M}$ is simply connected and $Q_\mathbb{M}$ splits as a direct sum $Q_\mathbb{M} = Q' \oplus Q''$, 
	then there exists topological $4$-manifolds $\mathbb{N}'$ and $\mathbb{N}''$ with intersection form
	$Q'$ and $Q''$ such that $\mathbb{M} \cong \mathbb{N}' \# \mathbb{N}''$. 
\end{cor}
%
Thus, the K3 surface (which is prime as PL-manifold) can be expressed as a connected sum of the form
$ 2 \,\mathbb{M}_{E_8} \# 3 \, S^2 \times S^2$.

The following result about simply connected PL $4$-manifolds 
is a combination of Theorems 1.2.21, 1.2.30 and 1.2.31 in \cite{Gompf} due to 
work by Rohlin \cite{Rohlin84NewResults4Mflds}, Milnor and Husemoller \cite{Milnor73SymmBilForms}, 
Donaldson \cite{Donaldson83GaugeTheory4Mflds}, and Furuta \cite{Furuta01MonopoleEq}.
\begin{prop}
	\label{prop:iforms}
	Suppose that $Q$ is the intersection form of a simply connected PL $4$-manifold.
	Then if $Q$ is odd, it is isomorphic to $k [ 1 ] \oplus l [-1]$, and if
	$Q$ is even it is isomorphic to $2m E_8 \oplus r H$, for some integers $k,l,r \geq 0$,
	$m \in \mathbb{Z}$, $r \geq 2 | m | + 1$.
\end{prop}
In particular, if $Q$ is even, it can not be definite. Furthermore, we have the following conjecture
\begin{conj}[$11/8$-conjecture \cite{Matsumoto82ElevenEight}]
	\label{conj:118}
	If $Q$, even, is the intersection form of a simply connected PL $4$-manifold.
	Then $Q \cong 2m E_8 \oplus r H$, $r \geq 0$, $m \in \mathbb{Z}$, $r \geq 3 | m |$.
\end{conj}	

In other words, Conjecture~\ref{conj:118} states that the rank of any even intersection
form admitting a PL structure is at least $11/8$ times as large as its signature, hence
the name. Assuming Conjecture~\ref{conj:118} is true it follows that all simply connected 
PL $4$-manifolds are homeomorphic to either 
$k \,\mathbb{CP}^2 \,\, \#  \,\, l \,\overline{\mathbb{CP}^2}$ or 
$\tilde{m} \, K3 \,\, \# \,\, \tilde{r} \, S^2 \times S^2$.

\medskip 
For complex hypersurfaces
$S_d = \{[z_0 : z_1 : z_2 : z_3] \in \mathbb{CP}^3 |  \sum {z_i}^d = 0\} \subset \mathbb{CP}^3$,
where $d$ is a positive integer, we have a similar result.

\begin{prop}[Theorem 1.3.8 \cite{Gompf}, McDuff and Salamon \cite{McDuff95SymplecticTop}]
The hypersurface $S_d$ is a PL, simply connected, complex surface. If $d$ is odd, 
then $Q_{S_d}$ is equivalent to $\lambda_d [1] \oplus \mu_d (-[1])$, 
where $\lambda_d = \frac{1}{3}(d^3 - 6 d^2 +11d -3)$ and $\mu_d =\frac{1}{3}(d-1)(2d^2-4d+3)$;
if $d$ is even, then $Q_{S_d}$ is equivalent to $l_d (-E_8) \oplus m_d H$, 
where $l_d = \frac{1}{24}d(d^2-4)$ and $m_d =\frac{1}{3}(d^3-6d^2+11d-3)$.
\end{prop}

%

\section{Simple crystallizations of simply connected PL 4-manifolds}
\label{sec:simpCryst}

Let $M$ be a contracted pseudotriangulation of a (simply connected) $4$-manifold $\mathbb{M}$, 
then $M$ has at least $\binom{5}{2}=10$ edges with equality if and only if
$M$ is simple.
%
%
In addition, since simple crystallizations by construction always describe simply
connected $4$-manifolds, and simply connected $4$-manifolds are always orientable,
it follows from Proposition~\ref{prop:ca-ga-pe80} that all simple crystallizations
are bipartite. Alternatively, this can also be followed from the fact that 
for any simple crystallization of a $4$-manifold $(\Gamma,\gamma)$ with color-set $C$, we have that 
$\Gamma_{C\setminus \{c\}}$ is a crystallization of $S^3$ for all $c \in C$. Thus,
$\Gamma_{C\setminus \{c\}}$ is bipartite for all $c \in C$.
Furthermore, $\Gamma_{\{i,j,k \}}$ is connected for all triples $i,j,k \in C$,
and hence the bipartite partition of $\Gamma$ is already fixed by the choice of 
any three colors. Hence, if $\Gamma$ is not bipartite, then there exists a color $c$ such 
that $\Gamma_{C\setminus \{c\}}$ is not bipartite, contradiction.

This property will be very useful in later sections of this article
and we will think of all simple crystallizations as bipartite graphs.
More generally, for arbitrary contracted simplicial cell complexes we have.

\begin{lemma} \label{Lemma:g=1}
	Let $X$ be a $d$-dimensional contracted simplicial cell complex and let $(\Gamma,\gamma)$ 
	be the $(d+1)$-colored graph corresponding to $X$ with color set $C$. Then
	$X$ is $k$-simple if and only if for all subsets $D \subset C$ of size $|d-k|$ the subgraph
	$\Gamma_{D}$ is connected.
\end{lemma}

\begin{proof}
	Let $X$ be $k$-simple, i.e., any set of $k+1$ vertices is in a unique $k$-cell.
	Let $D \subset C$ be of size $(d-k)$. Let $C\setminus D = \{j_1,\dots,j_{k+1}\}$ and
	let $v_{q}$ be the vertex of $X$ corresponding to the color $j_q$ 
	for $1\leq q \leq k+1$. Now, by construction, for each connected component
	of $\Gamma_{D}$ there is a distinct $k$-cell through 
	$v_1, \dots, v_{k+1}$ in $X$. Since $X$ is $k$-simple it follows
	that $\Gamma_{D}$ has exactly one component.
	
	Conversely, suppose $\Gamma_{D}$ is connected for all $D \subset C$ of size $d-k$. Let $v_1, \dots, v_{k+1}$ be $k+1$ vertices of $X$. Let $D=C \setminus \{j_1,\dots,j_{k+1}\}$ where $j_q$ is the color corresponding to the vertex $v_{q}$.
Then, by assumption, $\Gamma_{D}$ be connected.
This implies, number of $k$-cells through $v_1, \dots, v_{k+1}$ is one. This proves that $X$ is $k$-simple.
\end{proof}

\begin{lemma}\label{Lemma:connected-sum}
	Let $\mathbb{M}$ and $\mathbb{M}^{\prime}$ be closed connected $d$-manifolds
	each admitting a $k$-simple crystallization, then
	$\mathbb{M} \# \mathbb{M}^{\prime}$ admits a $k$-simple crystallization
	as well.
\end{lemma}

\begin{proof}
	Let $(\Gamma,\gamma)$ and $(\Gamma^{\prime},\gamma^{\prime})$ be $k$-simple crystallizations
	of $\mathbb{M}$ and $\mathbb{M}^{\prime}$ respectively with the same color set $C$. 

	Then by Lemma \ref{Lemma:g=1} both $\Gamma_{D}$ and $\Gamma^{\prime}_{D}$ are connected for all $D \subset C$, $\# D = d-k$.
	Hence, for any pair of vertices $(v_1,v_2) \in (V(\Gamma),V(\Gamma^{\prime}))$,
	$(\Gamma \#_{v_1v_2}\Gamma^{\prime})_{D}$ is connected and by applying Lemma \ref{Lemma:g=1}
	again it follows that $\Gamma \#_{v_1v_2}\Gamma^{\prime}$ is $k$-simple.
\end{proof}

In this article we are interested in simple contracted pseudotriangulations of closed $4$-manifolds
(which, then, are simply connected by construction, cf. Lemma~\ref{Lemma:simple-b}). 
In this special case we have the following.

\begin{lemma} \label{Lemma:f-vector}
	Let $M$ be a simple contracted pseudotriangulation of a closed $4$-manifold, 
	and let $(\Gamma, \gamma)$ be its corresponding crystallization 
	with color set $C = \{0,1,2,3,4\}$. Then
	$g_{ij}=m$ for any $2$-color set $\{i,j\} \subset C$,
	where $6m-4 = f_4 (M)$.
\end{lemma}

\begin{proof}
	Let $n=f_4 (M)$. Then $\Gamma$ has $n$ vertices.
	Since $M$ is simple, $\Gamma_{D}$ is connected for any $3$-color set 
	$D \subset C$ by Lemma \ref{Lemma:g=1}. Thus, $\Gamma_{E}$ 
	is a crystallization of $S^3$ for any $4$-color set $E \subset C$. 
	Now, by Proposition \ref{prop:gagliardi79a} (i), we have $m=g_{ij}=g_{kl}$ for 
	any $4$-color set $E = \{i,j,k, l\} \subset C$ and hence $g_{ij}=m$
	for any pair of colors $i,j \in C$. 
	In particular, by Proposition~\ref{prop:gagliardi79a} (ii), we have 
	$3m = g_{01}+g_{02}+g_{03} = n/2 + 2$ since $\Gamma_{\{0,1,2,3\}}$ 
	is a crystallization of $S^3$. This implies $3m =n/2+2$, i.e., $n=6m-4$. 
\end{proof}

\begin{remark}
	\label{rem:DSE}
	In fact, we can even state a more precise connection between the topology of
	a simply connected $4$-manifold and its simple contracted pseudotriangulations.

	Let $M$ be an $n$-facet contracted pseudotriangulation of a simply connected $4$-manifold
	$\mathbb{M}$. By construction, we have $f_0 (M) = 5$ and $\chi (\mathbb{M}) =
	2 + \beta_2 (\mathbb{M};\mathbb{Z})$. Furthermore, the Dehn-Sommerville equations
	in dimension four are as follows:
	$$
	\begin{array}{lllllllllll}
		f_0 (M) &-& f_1(M) &+& f_2(M) &-& f_3(M) &+& f_4(M) &=& \chi(M) \\
		&& 2 f_1(M) &-& 3 f_2(M) &+& 4 f_3(M) &-& 5 f_4(M) &=& 0 \\
		&&&&&& 2 f_3(M) &-& 5 f_4(M) &=& 0. \\
	\end{array}
	$$
	Replacing $f_0(M)$, $f_4 (M)$ and $\chi (M)$ and solving for $f_1$ in the above we get
	$$ 2 f_1 = n + 18 - 6 \beta_2 (\mathbb{M}; \mathbb{Z}). $$
	Now it immediately follows that $M$ is simple (that is, $f_1 = 10$) 
	if and only if $6 \beta_2 (\mathbb{M}; \mathbb{Z}) + 2 = f_4 (M)$. Furthermore, in this
	case we have $m = \beta_2 (\mathbb{M}; \mathbb{Z}) + 1$ due to Lemma~\ref{Lemma:f-vector}.

%
	In particular, we have for the number of vertices (facets) of simple crystallizations (simple contracted 
	pseudotriangulations) of the following manifolds:

	\begin{center}
	  \begin{tabular}{|l|c|c|}
		\hline
		Manifold&$m$&\# vertices / facets \\
		\hline
		\hline
		$S^4$&$1$&$2$ \\
		\hline
		$\mathbb{CP}^2$&$2$&$8$ \\
		\hline
		$S^2 \times S^2$&$3$&$14$ \\
		\hline
		$K3$&$23$&$134$ \\
		\hline
	  \end{tabular}
	\end{center}
\end{remark}

\section{The unique simple crystallizations of \boldmath{$S^4$} and \boldmath{$\mathbb{CP}^2$}}
\label{eg:cp2}

\subsection{The standard crystallization of $S^4$}
\label{ssec:s4}

The standard $2$-facet contracted pseudotriangulation of $S^4$ is given by gluing two 
$4$-dimensional simplices together along their
boundaries. The resulting complex is clearly a pseudotriangulation of $S^4$. Moreover, it has five vertices, 
ten edges, and is thus simple. It's crystallization is a $2$-vertex graph $\Gamma^0$ with five edges 
between two vertices and thus $\Gamma^0$ has a natural $5$-coloring (cf. Figure~\ref{fig:S4}).

\begin{figure}[ht]
\tikzstyle{vert}=[circle, draw, fill=black!100, inner sep=0pt, minimum width=6pt]
\tikzstyle{vertex}=[circle, draw, fill=black!00, inner sep=0pt, minimum width=6pt]
\tikzstyle{ver}=[]
\tikzstyle{edge} = [draw,thick,-]
\centering
\begin{tikzpicture}[scale=0.25]
\node[ver, inner sep=0pt, label={$0$}] (1) at (0,-4){};
\node[ver, inner sep=0pt, label={$1$}] (2) at (0,-2){};
\node[ver, inner sep=0pt, label={$2$}] (3) at (0,0){};
\node[ver, inner sep=0pt, label={$3$}] (4) at (0,2){};
\node[ver, inner sep=0pt, label={$4$}] (5) at (0,4){};
\node[vertex, label=left:{}] (10) at (-10,0){};
\node[vert, label=right:{}] (11) at (10,0){};

\draw[edge] plot [smooth,tension=1] coordinates{(10)(1)(11)};
\draw [line width=2pt, line cap=rectengle, dash pattern=on 1pt off 1] plot [smooth,tension=1] coordinates{(10)(1)(11)};
\draw[edge] plot [smooth,tension=1] coordinates{(10)(2)(11)};
\draw [line width=3pt, line cap=round, dash pattern=on 0pt off 2\pgflinewidth] plot [smooth,tension=1] coordinates{(10)(2)(11)};
\draw[edge] plot [smooth,tension=1] coordinates{(10)(11)};
\draw[edge, dotted] plot [smooth,tension=1] coordinates{(10)(4)(11)};
\draw[edge, dashed] plot [smooth,tension=1] coordinates{(10)(5)(11)};
\end{tikzpicture}

\caption{The standard crystallization of $S^4$.}
\label{fig:S4}
\end{figure}
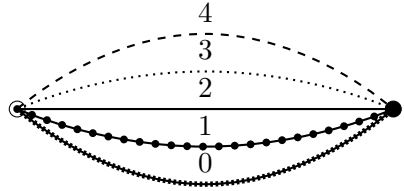

Note that the standard contracted pseudotriangulation of $S^4$ is the union of two standard $4$-balls glued together
along the boundaries of two $4$-simplexes. Hence, the 
standard contracted pseudotriangulation of $S^4$ is PL-homeomorphic to $S^4$ with standard PL structure. Moreover,
it is unique since there is no other $5$-colorable $5$-valent multi graph with only two vertices.
Hence, a version of Proposition~\ref{prop:pezzana74} for simple crystallizations of the $4$-sphere would
proof the Smooth Poincar\'e conjecture.

\subsection{The standard crystallization of \boldmath{$\mathbb{CP}^2$}}
\label{ssec:cp2}

The following example of a simple crystallization of $\mathbb{CP}^2$ first appeared in
\cite{Gagliardo89CP2}. 

Let $(\Gamma^1, \gamma^1)$ be the contracted $5$-colored graph with
color set $C=\{0,1,2,3,4\}$ given in Figure \ref{fig:G1} and let $M_1$ be 
its corresponding contracted pseudotriangulation. 
Since for every $2$-color subgraph $\Gamma^1_{\{i_1,i_2 \}}$, $\{i_1,i_2\} \subset C$, 
we have $g_{i_1 i_2}=2$ it follows from Proposition \ref{prop:gagliardi79a}
that $(\Gamma^1|_{C \setminus \{c\}}, \gamma^1|_{(\gamma^1)^{-1}(C \setminus \{c\})})$ is a crystallization 
of a $3$-manifold $\mathbb{M}_1^{(c)}$ for all $c \in C$. By Proposition \ref{prop:gagliardi79b} it is easy to calculate that
$\pi_1(\mathbb{M}_1^{(c)}, x) = \{0\}$, $x \in \mathbb{M}_1^{(c)}$, for all $c\in C$. 
Hence, due to Perelman's theorem \cite{Perelman07PC} 
$\mathbb{M}_1^{(c)} \cong S^3$ for all $c \in C$. Thus $(\Gamma^1, \gamma^1)$ is a crystallization and $M_1$ 
is a pseudotriangulation of a $4$-manifold $\mathbb{M}_1$. 
Since, in addition, any $3$-color subgraph of $(\Gamma^1, \gamma^1)$ is connected we get by Lemma \ref{Lemma:g=1} that
$M_1$ is a simple contracted pseudotriangulation and hence $\mathbb{M}_1$ is simply connected.
Furthermore, since $g_{i_1 i_2}=m=2$ the intersection form of
$\mathbb{M}_1$ has rank one, and by Freedman's classification theorem (cf. Theorem~\ref{Freedman1}) 
we know that  $\mathbb{M}_1$ must be homeomorphic to $\mathbb{CP}^2$.

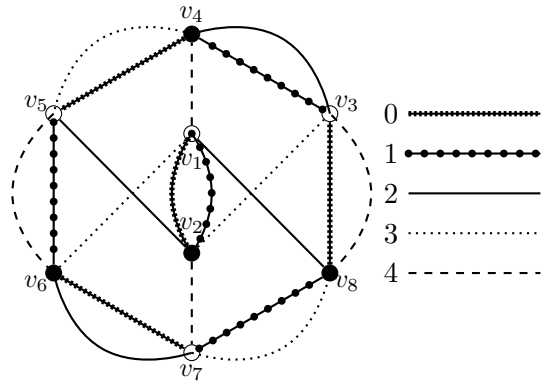
\begin{figure}[ht]
\tikzstyle{vert}=[circle, draw, fill=black!100, inner sep=0pt, minimum width=6pt]
\tikzstyle{vertex}=[circle, draw, fill=black!00, inner sep=0pt, minimum width=6pt]
\tikzstyle{ver}=[]
\tikzstyle{extra}=[circle, draw, fill=black!50, inner sep=0pt, minimum width=2pt]
\tikzstyle{edge} = [draw,thick,-]
\centering
\begin{tikzpicture}[scale=0.53]
\begin{scope}[rotate=30]
\foreach \x/\y in {0/3,120/5,240/7}{
\node[ver] (\y) at (\x:4.5){\small{$v_{\y}$}};
\node[vertex] (\y) at (\x:4){};}

\foreach \x/\y in {60/4,180/6,300/8}{
\node[ver] (\y) at (\x:4.5){\small{$v_{\y}$}};
\node[vert] (\y) at (\x:4) {};}

\foreach \x/\y in {3/4,5/6,7/8,8/3,4/5,6/7}{
\path[edge] (\x) -- (\y);}
\foreach \x/\y in {3/4,5/6,7/8}{
\draw [line width=3pt, line cap=round, dash pattern=on 0pt off 2\pgflinewidth] 	(\x) -- (\y);
}
\foreach \x/\y in {8/3,4/5,6/7}{ 
\draw[line width=2pt, line cap=rectengle, dash pattern=on 1pt off 1]  (\x) -- (\y);}
\end{scope}
\begin{scope}
\foreach \x/\y in {0/9,120/11,240/13}{
\node[ver] (\y) at (\x:4.5){};}
\foreach \x/\y in {60/10,180/12,300/14}{
\node[ver] (\y) at (\x:4.5) {};}
\node[vertex] (1) at (0,1.5){};
\node[vert] (2) at (0,-1.5){};
\node[ver] () at (0,0.8){\small{$v_1$}};
\node[ver] () at (0,-0.8){\small{$v_2$}};
\end{scope}
\foreach \x/\y/\z in {3/4/10,6/7/13}{
\draw[edge] plot [smooth,tension=1] coordinates{(\x)(\z)(\y)};}
\foreach \x/\y/\z in {4/5/11,7/8/14}{
\draw[edge, dotted] plot [smooth,tension=1] coordinates{(\x)(\z)(\y)};}
\foreach \x/\y/\z in {5/6/12,8/3/9}{
\draw[edge, dashed] plot [smooth,tension=1] coordinates{(\x)(\z)(\y)};}
\draw[edge] plot [smooth,tension=1] coordinates{(1)(-0.5,0)(2)};
\draw[edge] plot [smooth,tension=1] coordinates{(1)(0.5,0)(2)};
\draw[line width=2pt, line cap=rectengle, dash pattern=on 1pt off 1] plot [smooth,tension=1] coordinates{(1)(-0.5,0)(2)};
\draw[line width=3pt, line cap=round, dash pattern=on 0pt off 2\pgflinewidth] plot [smooth,tension=1] coordinates{(1)(0.5,0)(2)};
\path[edge, dashed] (1) -- (4);
\path[edge, dotted] (1) -- (6);
\path[edge] (1) -- (8);
\path[edge, dotted] (2) -- (3);
\path[edge] (2) -- (5);
\path[edge, dashed] (2) -- (7);
\begin{scope}[shift={(7,2)}]
\node[ver] (300) at (-2,0){$0$};
\node[ver] (301) at (-2,-1){$1$};
\node[ver] (302) at (-2,-2){$2$};
\node[ver] (303) at (-2,-3){$3$};
\node[ver] (304) at (-2,-4){$4$};
\node[ver] (305) at (2,0){};
\node[ver] (306) at (2,-1){};
\node[ver] (307) at (2,-2){};
\node[ver] (308) at (2,-3){};
\node[ver] (309) at (2,-4){};

\path[edge] (300) -- (305);
\draw [line width=2pt, line cap=rectengle, dash pattern=on 1pt off 1]  (300) -- (305);
\path[edge] (301) -- (306);
\draw [line width=3pt, line cap=round, dash pattern=on 0pt off 2\pgflinewidth]  (301) -- (306);
\path[edge] (302) -- (307);
\path[edge, dotted] (303) -- (308);
\path[edge, dashed] (304) -- (309);
\end{scope}
\end{tikzpicture}
\caption{A simple crystallization of $\mathbb{CP}^2$.}
\label{fig:G1}
\end{figure}

\begin{remark}
	There is another elegant way to construct $(\Gamma^1, \gamma^1)$
	(cf. \cite{Chiavacci93LinkingMinTrig}, where a similar construction in
	terms of cancelling dipoles is presented): 
	Given the unique $3$-neighborly $9$-vertex triangulation of the
	(standard PL) complex projective plane \cite{Kuehnel83The9VertComplProjPlane}, 
	there is a sequence of four edge contractions
	transforming it to $(\Gamma^1, \gamma^1)$. This is the minimum number of elementary moves necessary 
	to pass from a $9$-vertex combinatorial manifold to a $5$-vertex simple contracted pseudotriangulation. 
	In particular, this shows that $(\Gamma^1, \gamma^1)$ is of standard PL type.
	The facet lists of the five pseudotriangulations are available from the authors upon request.

\end{remark}

\subsection{Uniqueness of the simple crystallization of \boldmath{$\mathbb{CP}^2$}}

From the classification of crystallizations of $4$-manifolds 
\cite{Casali12Crystallizations,CasaliDUKEIII,Casali14Cataloguing} we get that there is
exactly one simple crystallization with eight vertices. In particular, it follows
that there is exactly one simple crystallization of $\mathbb{CP}^2$. In this section we 
give an elementary proof of the uniqueness of $(\Gamma^1, \gamma^1)$ independent
of the classification of crystallizations of $3$- and $4$-manifolds.

\newpage

\begin{theo}
	\label{thm:unique}
	Up to isomorphy, $(\Gamma^1, \gamma^1)$ is the only simple crystallization of $\mathbb{CP}^2$.
\end{theo}

\begin{proof}
Let $(\Gamma, \gamma)$ be a simple crystallization of $\mathbb{CP}^2$ with color set $C=\{0,1,2,3,4\}$. Then
$\Gamma$ must have eight vertices, $g_{i j}=2$ for $\{i,j\} \subset C$, 
and $(\Gamma|_{C \setminus \{c\}}, \gamma|_{(\gamma^1)^{-1}(C \setminus \{c\})})$ 
is a crystallization of $S^3$ for all $c \in C$. Since $\Gamma$ is simple 
we have $g_{ijk}=1$ for $\{i,j,k\} \subset C$ by Lemma~\ref{Lemma:g=1} 
and hence $\Gamma$ can not contain a triple edge (Otherwise, 
let the three edges between two vertices be colored by $i$, $j$ and $k$. Then the triple edge
on its own must be a connected component of $\Gamma_{\{ i,j,k \}}$ and we have 
$g_{ijk}>1$ whenever $\Gamma$ has more than two vertices). Furthermore, since $\mathbb{CP}^2$ 
is orientable, $\Gamma$ must be bipartite by Proposition~\ref{prop:ca-ga-pe80}. In particular,
$\Gamma$ can not have any cycles of odd length. It follows that $\Gamma_{\{ i,j\}}$, $\{i,j\} \subset C$,
must be of the form $C_2 \sqcup C_6$ or $C_4 \sqcup C_4$.

In the following we will use this fact to prove that there exist precisely three crystallizations 
of $S^3$ which can occur as a $4$-color subgraph of a simple crystallization of $\mathbb{CP}^2$.
The theorem then follows from the fact that there is a unique $5$-colored graph such that 
all of its $4$-color subgraphs are isomorphic to one of these three crystallizations.

\begin{figure}[ht]
\tikzstyle{vert}=[circle, draw, fill=black!100, inner sep=0pt, minimum width=4pt]
\tikzstyle{vertex}=[circle, draw, fill=black!00, inner sep=0pt, minimum width=4pt]
\tikzstyle{ver}=[]
\tikzstyle{extra}=[circle, draw, fill=black!50, inner sep=0pt, minimum width=2pt]
\tikzstyle{edge} = [draw,thick,-]

\centering
\begin{tikzpicture}[scale=0.37]

\begin{scope}[shift={(-10,7)}]
\node[vertex] (a1) at (-1.5,-1.5){}; 
\node[vert] (a4) at (-1.5,1.5){}; 
\node[vert] (a2) at (1.5,-1.5){};
\node[vertex] (a3) at (1.5,1.5){}; 

\node[vert] (b1) at (-4,-4){}; 
\node[vertex] (b4) at (-4,4){};
\node[vertex] (b2) at (4,-4){}; 
\node[vert] (b3) at (4,4){};

\foreach \x/\y in {a1/a2,a2/a3,a3/a4,a4/a1,b1/b2,b2/b3,b3/b4,b4/b1}{ \path[edge] (\x) -- (\y);} 
\foreach \x/\y in {a1/a2,a3/a4,b1/b2,b4/b3}{ 
\draw [line width=2pt, line cap=rectengle, dash pattern=on 1pt off 1] (\x) -- (\y);}
\foreach \x/\y in {a3/a2,a1/a4,b3/b2,b1/b4}{ 
\draw [line width=3pt, line cap=round, dash pattern=on 0pt off 2\pgflinewidth]   (\x) -- (\y);} 
\foreach \x/\y in {a1/b1,a2/b2,a3/b3,a4/b4}{ 
\path[edge] (\x) -- (\y);} 
\draw[edge, dotted] plot [smooth,tension=0.5] coordinates{(a1)(2.5,-2.5)(b3)}; 
\draw[edge, dotted] plot [smooth,tension=0.5] coordinates{(a2)(2.5,2.5)(b4)}; 
\draw[edge, dotted] plot [smooth,tension=0.5] coordinates{(a3)(-2.5,2.5)(b1)};
\draw[edge, dotted] plot [smooth,tension=0.5] coordinates{(a4)(-2.5,-2.5)(b2)}; 
\node[ver] () at (-0.8,-0.8){$a_1$}; 
\node[ver] () at (-0.8,0.8){$a_4$}; 
\node[ver] () at (0.8,-0.8){$a_2$}; 
\node[ver] () at (0.8,0.8){$a_3$}; 
\node[ver] () at (-4.5,-4.7){$b_1$}; 
\node[ver] () at (-4.5,4.7){$b_4$}; 
\node[ver] () at (4.5,-4.7){$b_2$}; 
\node[ver] () at (4.5,4.7){$b_3$};
\node[ver] (208) at (0,-6){$(a)$ Crystallization of $\mathbb{RP}^3$};
\end{scope}

\begin{scope}[shift={(1,7)}]
\node[vertex] (107) at (-2.5,-4){};
\node[vertex] (105) at (-2.5,1.3){};
\node[vertex] (103) at (2.5,-1.3){};
\node[vertex] (102) at (2.5,4){};

\node[vert] (108) at (2.5,-4){};
\node[vert] (104) at (2.5,1.3){};
\node[vert] (106) at (-2.5,-1.3){};
\node[vert] (101) at (-2.5,4){};

\path[edge] (101) -- (105);
\path[edge] (102) -- (104);
\path[edge] (103) -- (106);

\path[edge] (106) -- (105);
\draw [line width=3pt, line cap=round, dash pattern=on 0pt off 2\pgflinewidth] (106) -- (105);
\path[edge] (103) -- (104);
\draw [line width=3pt, line cap=round, dash pattern=on 0pt off 2\pgflinewidth] (103) -- (104);

\path[edge] (108) -- (103);
\draw [line width=2pt, line cap=rectengle, dash pattern=on 1pt off 1] (103) -- (108);
\path[edge] (107) -- (106);
\draw [line width=2pt, line cap=rectengle, dash pattern=on 1pt off 1] (106) -- (107);
\path[edge] (104) -- (105);
\draw [line width=2pt, line cap=rectengle, dash pattern=on 1pt off 1] (104) -- (105);

\draw [line width=2pt, line cap=rectengle, dash pattern=on 1pt off 1] plot [smooth,tension=1.5] coordinates{(101)(0,4.4)(102)};
\draw [line width=3pt, line cap=round, dash pattern=on 0pt off 2\pgflinewidth] plot [smooth,tension=1.5] coordinates{(101)(0,3.6)(102)};
\draw[edge] plot [smooth,tension=1.5] coordinates{(101)(0,4.4)(102)};
\draw[edge] plot [smooth,tension=1.5] coordinates{(101)(0,3.6)(102)};

\draw [line width=3pt, line cap=round, dash pattern=on 0pt off 2\pgflinewidth] plot [smooth,tension=1.5] coordinates{(107)(0,-4.4)(108)};
\draw[edge] plot [smooth,tension=1.5] coordinates{(107)(0,-4.4)(108)};
\draw[edge] plot [smooth,tension=1.5] coordinates{(107)(0,-3.6)(108)};

\node[ver] (200) at (3.3,-1.3){$v_3$};
\node[ver] (201) at (3.3,1.3){$v_4$};
\node[ver] (202) at (3.3,4.5){$v_2$};
\node[ver] (203) at (-3.3,4.5){$v_1$};
\node[ver] (204) at (-3.3,1.3){$v_5$};
\node[ver] (205) at (-3.3,-1.3){$v_6$};
\node[ver] (206) at (-3.3,-4.8){$v_7$};
\node[ver] (207) at (3.3,-4.8){$v_8$};
\node[ver] (208) at (0,-6){$(b)$ The graph $\mathcal{J}_1$};
\end{scope}

\begin{scope}[shift={(12,7)}]
\node[vertex] (107) at (-2,-4){};
\node[vertex] (105) at (-2,4){};
\node[vertex] (103) at (4.5,0){};
\node[vertex] (102) at (1,-2){};

\node[vert] (108) at (2,-4){};
\node[vert] (104) at (2,4){};
\node[vert] (106) at (-4.5,0){};
\node[vert] (101) at (-1,2){};

\path[edge] (101) -- (105);
\path[edge] (102) -- (108);

\draw[edge] (107) -- (108);

\draw [line width=3pt, line cap=round, dash pattern=on 0pt off 2\pgflinewidth] (107) -- (108);
\path[edge] plot [smooth,tension=1.5] coordinates{(101)(0.3,0.3)(102)};
\path[edge] plot [smooth,tension=1.5] coordinates{(101)(-0.3,-0.3)(102)};

\path[edge] (106) -- (105);
\draw [line width=3pt, line cap=round, dash pattern=on 0pt off 2\pgflinewidth] (106) -- (105);

\path[edge] (108) -- (103);
\draw [line width=2pt, line cap=rectengle, dash pattern=on 1pt off 1] (103) -- (108);
\path[edge] (104) -- (105);
\draw [line width=2pt, line cap=rectengle, dash pattern=on 1pt off 1] (104) -- (105);

\draw [line width=2pt, line cap=rectengle, dash pattern=on 1pt off 1] plot [smooth,tension=1.5] coordinates{(101)(0.3,0.3)(102)};
\draw [line width=3pt, line cap=round, dash pattern=on 0pt off 2\pgflinewidth] plot [smooth,tension=1.5] coordinates{(101)(-0.3,-0.3)(102)};

\path[edge] plot [smooth,tension=1.5] coordinates{(106)(-2.95,-1.7)(107)};
\path[edge] plot [smooth,tension=1.5] coordinates{(106)(-3.55,-2.3)(107)};
\draw [line width=2pt, line cap=rectengle, dash pattern=on 1pt off 1] plot [smooth,tension=1.5] coordinates{(106)(-3.55,-2.3)(107)};

\path[edge] plot [smooth,tension=1.5] coordinates{(103)(3.55,2.3)(104)};
\draw [line width=3pt, line cap=round, dash pattern=on 0pt off 2\pgflinewidth] plot [smooth,tension=1.5] coordinates{(103)(3.55,2.3)(104)};
\path[edge] plot [smooth,tension=1.5] coordinates{(103)(2.95,1.7)(104)};

\node[ver] (200) at (5.5,0){$v_3$};
\node[ver] (201) at (3,4.5){$v_4$};
\node[ver] (202) at (1.8,-2.2){$v_2$};
\node[ver] (203) at (-1.8,2.2){$v_1$};
\node[ver] (204) at (-3,4.5){$v_5$};
\node[ver] (205) at (-5.5,0){$v_6$};
\node[ver] (206) at (-3,-4.8){$v_7$};
\node[ver] (207) at (3,-4.8){$v_8$};
\node[ver] (208) at (0,-6){$(c)$ The graph $\mathcal{J}_2$};
\end{scope}

\begin{scope}[shift={(22,5)}]
\node[ver] (308) at (-3,5){$0$};
\node[ver] (300) at (-3,3){$1$};
\node[ver] (301) at (-3,1){$2$};
\node[ver] (302) at (-3,-1){$3$};
\node[ver] (309) at (3,5){};
\node[ver] (304) at (3,3){};
\node[ver] (305) at (3,1){};
\node[ver] (306) at (3,-1){};
\path[edge] (300) -- (304);
\path[edge] (308) -- (309);
\draw [line width=2pt, line cap=rectengle, dash pattern=on 1pt off 1]  (308) -- (309);
\draw [line width=3pt, line cap=round, dash pattern=on 0pt off 2\pgflinewidth]  (300) -- (304);
\path[edge] (301) -- (305);
\path[edge, dotted] (302) -- (306);
\end{scope}

\begin{scope}[shift={(-10,-7)}]
\node[vertex] (107) at (-2.5,-4){};
\node[vertex] (105) at (-2.5,1.3){};
\node[vertex] (103) at (2.5,-1.3){};
\node[vertex] (102) at (2.5,4){};

\node[vert] (108) at (2.5,-4){};
\node[vert] (104) at (2.5,1.3){};
\node[vert] (106) at (-2.5,-1.3){};
\node[vert] (101) at (-2.5,4){};

\path[edge] (101) -- (105);
\path[edge] (102) -- (104);
\path[edge] (103) -- (106);

\path[edge] (106) -- (105);
\draw [line width=3pt, line cap=round, dash pattern=on 0pt off 2\pgflinewidth] (106) -- (105);
\path[edge] (103) -- (104);
\draw [line width=3pt, line cap=round, dash pattern=on 0pt off 2\pgflinewidth] (103) -- (104);

\path[edge] (108) -- (103);
\draw [line width=2pt, line cap=rectengle, dash pattern=on 1pt off 1] (103) -- (108);
\path[edge] (107) -- (106);
\draw [line width=2pt, line cap=rectengle, dash pattern=on 1pt off 1] (106) -- (107);
\path[edge] (104) -- (105);
\draw [line width=2pt, line cap=rectengle, dash pattern=on 1pt off 1] (104) -- (105);

\draw [line width=2pt, line cap=rectengle, dash pattern=on 1pt off 1] plot [smooth,tension=1.5] coordinates{(101)(0,4.4)(102)};
\draw [line width=3pt, line cap=round, dash pattern=on 0pt off 2\pgflinewidth] plot [smooth,tension=1.5] coordinates{(101)(0,3.6)(102)};
\draw[edge] plot [smooth,tension=1.5] coordinates{(101)(0,4.4)(102)};
\draw[edge] plot [smooth,tension=1.5] coordinates{(101)(0,3.6)(102)};

\draw [line width=3pt, line cap=round, dash pattern=on 0pt off 2\pgflinewidth] plot [smooth,tension=1.5] coordinates{(107)(0,-4.4)(108)};
\draw[edge] plot [smooth,tension=1.5] coordinates{(107)(0,-4.4)(108)};
\draw[edge] plot [smooth,tension=1.5] coordinates{(107)(0,-3.6)(108)};

\node[ver] (200) at (3.3,-1.3){$v_3$};
\node[ver] (201) at (3.3,1.3){$v_4$};
\node[ver] (202) at (3.3,4.5){$v_2$};
\node[ver] (203) at (-3.3,4.5){$v_1$};
\node[ver] (204) at (-3.3,1.3){$v_5$};
\node[ver] (205) at (-3.3,-1.3){$v_6$};
\node[ver] (206) at (-3.3,-4.8){$v_7$};
\node[ver] (207) at (3.3,-4.8){$v_8$};

\draw[edge, dotted] plot [smooth,tension=1.5] coordinates{(101)(4,5)(103)};
\draw[edge, dotted] plot [smooth,tension=1.5] coordinates{(102)(3.2,2.6)(104)};
\draw[edge, dotted] plot [smooth,tension=1] coordinates{(105)(-4,-5)(108)};
\draw[edge, dotted] plot [smooth,tension=1] coordinates{(106)(-3.2,-2.6)(107)};
\node[ver] (108) at (0,-6){$(d)$ Crystallization $\mathcal{G}_1$};
\end{scope}

\begin{scope}[shift={(2.5,-7)}]
\node[vertex] (107) at (-2.5,-4){};
\node[vertex] (105) at (-2.5,1.3){};
\node[vertex] (103) at (2.5,-1.3){};
\node[vertex] (102) at (2.5,4){};

\node[vert] (108) at (2.5,-4){};
\node[vert] (104) at (2.5,1.3){};
\node[vert] (106) at (-2.5,-1.3){};
\node[vert] (101) at (-2.5,4){};

\path[edge] (101) -- (105);
\path[edge] (102) -- (104);
\path[edge] (103) -- (106);

\path[edge] (106) -- (105);
\draw [line width=3pt, line cap=round, dash pattern=on 0pt off 2\pgflinewidth] (106) -- (105);
\path[edge] (103) -- (104);
\draw [line width=3pt, line cap=round, dash pattern=on 0pt off 2\pgflinewidth] (103) -- (104);

\path[edge] (108) -- (103);
\draw [line width=2pt, line cap=rectengle, dash pattern=on 1pt off 1] (103) -- (108);
\path[edge] (107) -- (106);
\draw [line width=2pt, line cap=rectengle, dash pattern=on 1pt off 1] (106) -- (107);
\path[edge] (104) -- (105);
\draw [line width=2pt, line cap=rectengle, dash pattern=on 1pt off 1] (104) -- (105);

\draw [line width=2pt, line cap=rectengle, dash pattern=on 1pt off 1] plot [smooth,tension=1.5] coordinates{(101)(0,4.4)(102)};
\draw [line width=3pt, line cap=round, dash pattern=on 0pt off 2\pgflinewidth] plot [smooth,tension=1.5] coordinates{(101)(0,3.6)(102)};
\draw[edge] plot [smooth,tension=1.5] coordinates{(101)(0,4.4)(102)};
\draw[edge] plot [smooth,tension=1.5] coordinates{(101)(0,3.6)(102)};

\draw [line width=3pt, line cap=round, dash pattern=on 0pt off 2\pgflinewidth] plot [smooth,tension=1.5] coordinates{(107)(0,-4.4)(108)};
\draw[edge] plot [smooth,tension=1.5] coordinates{(107)(0,-4.4)(108)};
\draw[edge] plot [smooth,tension=1.5] coordinates{(107)(0,-3.6)(108)};

\node[ver] (200) at (3.3,-1.3){$v_3$};
\node[ver] (201) at (3.3,1.3){$v_4$};
\node[ver] (202) at (3.3,4.5){$v_2$};
\node[ver] (203) at (-3.3,4.5){$v_1$};
\node[ver] (204) at (-3.3,1.3){$v_5$};
\node[ver] (205) at (-3.3,-1.3){$v_6$};
\node[ver] (206) at (-3.3,-4.8){$v_7$};
\node[ver] (207) at (3.3,-4.8){$v_8$};
\draw[edge, dotted] plot [smooth,tension=1.5] coordinates{(101)(4.5,5)(103)};
\draw[edge, dotted] plot [smooth,tension=1.5] coordinates{(102)(4,0)(108)};
\draw[edge, dotted] plot [smooth,tension=1] coordinates{(104)(4.5,-5)(107)};
\draw[edge, dotted] plot [smooth,tension=1] coordinates{(105)(-3.2,0)(106)};
\node[ver] (208) at (0,-6){$(e)$ Crystallization $\mathcal{G}_2$};
\end{scope}

\begin{scope}[shift={(16,-7)}]
\node[vertex] (107) at (-2.5,-4){};
\node[vertex] (105) at (-2.5,1.3){};
\node[vertex] (103) at (2.5,-1.3){};
\node[vertex] (102) at (2.5,4){};

\node[vert] (108) at (2.5,-4){};
\node[vert] (104) at (2.5,1.3){};
\node[vert] (106) at (-2.5,-1.3){};
\node[vert] (101) at (-2.5,4){};

\path[edge] (101) -- (105);
\path[edge] (102) -- (104);
\path[edge] (103) -- (106);

\path[edge] (106) -- (105);
\draw [line width=3pt, line cap=round, dash pattern=on 0pt off 2\pgflinewidth] (106) -- (105);
\path[edge] (103) -- (104);
\draw [line width=3pt, line cap=round, dash pattern=on 0pt off 2\pgflinewidth] (103) -- (104);

\path[edge] (108) -- (103);
\draw [line width=2pt, line cap=rectengle, dash pattern=on 1pt off 1] (103) -- (108);
\path[edge] (107) -- (106);
\draw [line width=2pt, line cap=rectengle, dash pattern=on 1pt off 1] (106) -- (107);
\path[edge] (104) -- (105);
\draw [line width=2pt, line cap=rectengle, dash pattern=on 1pt off 1] (104) -- (105);

\draw [line width=2pt, line cap=rectengle, dash pattern=on 1pt off 1] plot [smooth,tension=1.5] coordinates{(101)(0,4.4)(102)};
\draw [line width=3pt, line cap=round, dash pattern=on 0pt off 2\pgflinewidth] plot [smooth,tension=1.5] coordinates{(101)(0,3.6)(102)};
\draw[edge] plot [smooth,tension=1.5] coordinates{(101)(0,4.4)(102)};
\draw[edge] plot [smooth,tension=1.5] coordinates{(101)(0,3.6)(102)};

\draw [line width=3pt, line cap=round, dash pattern=on 0pt off 2\pgflinewidth] plot [smooth,tension=1.5] coordinates{(107)(0,-4.4)(108)};
\draw[edge] plot [smooth,tension=1.5] coordinates{(107)(0,-4.4)(108)};
\draw[edge] plot [smooth,tension=1.5] coordinates{(107)(0,-3.6)(108)};

\node[ver] (200) at (3.3,-1.3){$v_3$};
\node[ver] (201) at (3.3,1.3){$v_4$};
\node[ver] (202) at (3.3,4.5){$v_2$};
\node[ver] (203) at (-3.3,4.5){$v_1$};
\node[ver] (204) at (-3.3,1.3){$v_5$};
\node[ver] (205) at (-3.3,-1.3){$v_6$};
\node[ver] (206) at (-3.3,-4.8){$v_7$};
\node[ver] (207) at (3.3,-4.8){$v_8$};
\draw[edge, dotted] plot [smooth,tension=1.5] coordinates{(101)(-3.2,2.6)(105)};
\draw[edge, dotted] plot [smooth,tension=1.5] coordinates{(103)(3.2,0)(104)};
\draw[edge, dotted] plot [smooth,tension=1] coordinates{(106)(-3.2,-2.6)(107)};
\draw[edge, dotted] plot [smooth,tension=1] coordinates{(102)(4.5,0)(108)};
\node[ver] (208) at (0,-6){$(f)$ Crystallization $\mathcal{G}_3$};
\end{scope}

\end{tikzpicture}
\caption{The graphs $\mathcal{J}_i$, $1\leq i \leq 2$, and $\mathcal{G}_j$, $1\leq i \leq 3$. Note that all graphs
are bipartite since all edges go between vertices of type $\bullet$ and vertices of type $\circ$.}
\label{fig:S3inCP2}
\end{figure}
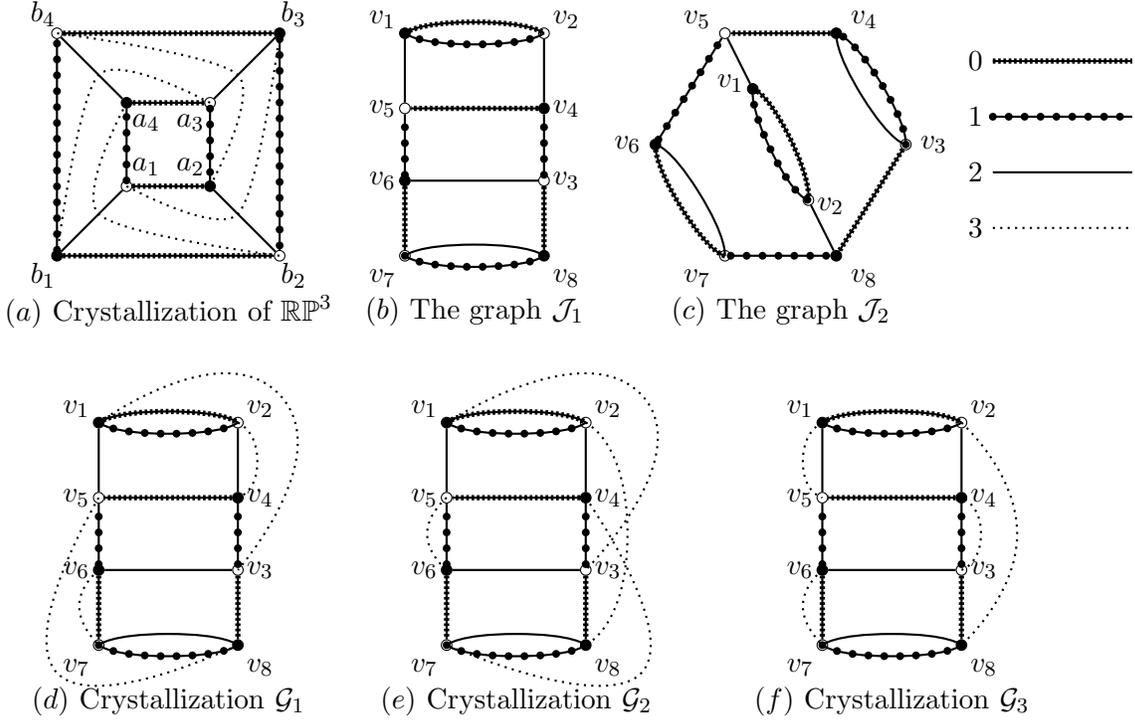

\bigskip
Let $\Gamma_{\{i,j\}}$ be of the form  $C_4 \sqcup C_4$ for all $\{ i,j \} \subset C$. 
That is, in particular, 
$\Gamma_{\{0,1\}}= C_4(a_1,a_2,a_3,a_4) \sqcup C_4(b_1,b_2,b_3,b_4)$ 
as in Figure \ref{fig:S3inCP2} $(a)$. Since $\Gamma_{\{0,1,2\}}$ must be
connected, without loss of generality let $a_1b_1 \in \gamma^{-1}(2)$. 
Since by assumption $\Gamma_{\{0,2\}}$ and $\Gamma_{\{1,2\}}$ are of 
the form $C_4 \sqcup C_4$ it follows (up to isomorphism) that $a_2b_2, a_3b_3, a_4b_4 \in \gamma^{-1}(2)$.
Now, since $\Gamma_{\{0,1,3\}}$ is connected, bipartite and $\Gamma$ does not contain a 
$2$-cycle, $a_1b_3, a_2b_4, a_3b_1, a_4b_2 \in \gamma^{-1}(3)$ (see Figure
\ref{fig:S3inCP2} ($a$)). By applying Proposition~\ref{prop:gagliardi79b}
we get that $\pi_1(|M(\Gamma_{\{0,1,2,3\}})|,x) \cong \mathbb{Z}_2$ 
and hence $\Gamma_{\{0,1,2,3\}}$ can not be a crystallization of $S^3$. 
Hence there exist $\{ i, j \} \subset C$ such that $\Gamma_{\{i,j\}}$ 
is of the form $C_2 \sqcup C_6$. Without loss of generality,
let $\Gamma_{\{0,1\}}= C_2(v_1,v_2) \sqcup C_6(v_3,v_4,v_5,v_6,v_7, v_8)$ 
as in Figure \ref{fig:S3inCP2} $(b)$, and since $\Gamma_{\{0,1,2\}}$
is connected, let $v_1v_5 \in \gamma^{-1}(2)$. 
Then either $v_2v_4 \in \gamma^{-1}(2)$ or $v_2v_8 \in \gamma^{-1}(2)$
(up to isomorphism). 

\bigskip
\noindent
{\bf Case $v_2v_4 \in \gamma^{-1}(2)$}: In this case we have $v_3v_6, v_7v_8 \in \gamma^{-1}(2)$ 
since $\Gamma_{\{1,2\}}$ has two connected components. We will denote this graph by 
$\mathcal{J}_1$ (see Figure \ref{fig:S3inCP2} $(b)$).

Since $\Gamma$ is bipartite and does not contain a triple edge, there are three possibilities
to add an edges of color $3$ to $\mathcal{J}_1$ at vertex $v_1$. 
\begin{enumerate}[(i)]
	\item Let $v_1v_3 \in \gamma^{-1}(3)$.
		Then $v_2v_6 \not \in \gamma^{-1}(3)$, since otherwise a triple edge has to be inserted at $v_7 v_8$
		to ensure that $\Gamma_{\{0,3\}}$ has two connected components. Hence either $v_2v_4 \in \gamma^{-1}(3)$ 
		or $v_2v_8 \in \gamma^{-1}(3)$. If $v_2v_4 \in \gamma^{-1}(3)$ then $v_6v_7, v_5v_8 \in \gamma^{-1}(3)$. 
		We will denote this graph by $\mathcal{G}_1$ (see Figure \ref{fig:S3inCP2} $(d)$). 
		If $v_2v_8 \in \gamma^{-1}(3)$ then $v_5v_6, v_4v_7 \in \gamma^{-1}(3)$. We will denote this graph 
		by $\mathcal{G}_2$ (see Figure \ref{fig:S3inCP2} $(e)$). 
	\item Let  $v_1v_5 \in \gamma^{-1}(3)$. First note that $v_2v_4 \not \in \gamma^{-1}(3)$ since otherwise 
		$g_{23} > 2$. If $v_2v_6 \in \gamma^{-1}(3)$, then $v_3 v_8, v_4 v_7 \in \gamma^{-1}(3)$
		and we get a graph isomorphic to $\mathcal{G}_1$. If $v_2v_8 \in \gamma^{-1}(3)$, then $v_3v_4, v_6v_7 \in \gamma^{-1}(3)$
		since otherwise $g_{13} < 2$. We will denote this graph by $\mathcal{G}_3$ (see Figure \ref{fig:S3inCP2} $(f)$).
	\item Let $v_1v_7 \in \gamma^{-1}(3)$. If $v_2v_4$ (resp., $v_2v_6$) $\in \gamma^{-1}(3)$, then we get graphs isomorphic to 
		$\mathcal{G}_3$ (resp., $\mathcal{G}_2$). Thus, assume $v_2v_8\in \gamma^{-1}(3)$ and hence 
		$v_4v_5, v_3v_6\in \gamma^{-1}(3)$ (otherwise $g_{13}>2$). In this case we have $\pi_1(|M(\Gamma_{\{0,1,2,3\}})|,x) \cong \mathbb{Z}$
		by Proposition~\ref{prop:gagliardi79b} and thus $\Gamma_{\{0,1,2,3\}}$ is not a crystallization of $S^3$.
\end{enumerate}

\bigskip
\noindent
{\bf Case $v_2v_8 \in \gamma^{-1}(2)$}: It follows that $v_3v_4, v_6v_7 \in \gamma^{-1}(2)$ 
since $\Gamma_{\{1,2\}}$ has two connected components. We will denote this graph by 
$\mathcal{J}_2$ (see Figure \ref{fig:S3inCP2} $(c)$).

Since $\Gamma$ is bipartite and does not contain a triple edge, there are two possibilities
(up to isomorphy) to add an edge of color $3$ to $\mathcal{J}_2$ at vertex $v_1$. 
\begin{enumerate}[(i)]
	\item Let $v_1v_3 \in \gamma^{-1}(3)$. First note that $v_2 v_4 \not \in \gamma^{-1}(3)$ since otherwise 
		$\Gamma_{\{0,1,2,3\}}$ would have a triple edge or $g_{13} >2$.
		If $v_2 v_6 \in \gamma^{-1}(3)$ then $v_3 v_4, v_7v_8 \in \gamma^{-1}(3)$ since otherwise 
		$g_{03}<2$ and $\Gamma_{\{0,1,2,3\}}$ is isomorphic to $\mathcal{G}_3$. If $v_2 v_8 \in \gamma^{-1}(3)$
		then $v_4 v_7, v_5v_6 \in \gamma^{-1}(3)$ to avoid a triple edge at $v_6 v_7$ and, again
		$\Gamma_{\{0,1,2,3\}}$ is isomorphic to $\mathcal{G}_3$.
	\item Let  $v_1v_5 \in \gamma^{-1}(3)$. We have $v_2 v_8 \not \in \gamma^{-1}(3)$ since otherwise 
		$g_{23} > 2$. Then, up to isomorphy we must have $v_2v_4 \in \gamma^{-1}(3)$ and thus $v_3 v_6, v_7v_8 \in \gamma^{-1}(3)$ 
		to avoid a triple edge and we get a graph isomorphic to $\mathcal{G}_3$.
\end{enumerate}

Hence, there are exactly three $8$-vertex crystallizations $\mathcal{G}_1$, $\mathcal{G}_2$ and $\mathcal{G}_3$ of $S^3$ such that $g_{ij} = 2$ for all $\{i,j\} \subset \{ 1,2,3,4 \}$.
This part of the proof can be independently checked using the classification of generalized triangulations \cite{Burton11Census}:
Precisely ten of the $9\,787\,509$ closed eight tetrahedra generalized triangulations of $3$-manifolds are
contracted pseudotriangulations, seven of which are crystallizations of the $3$-sphere, three of which satisfy $g_{ij} = 2$
for all $2$-color subsets $\{i,j\} \subset \{0,1,2,3\}$.


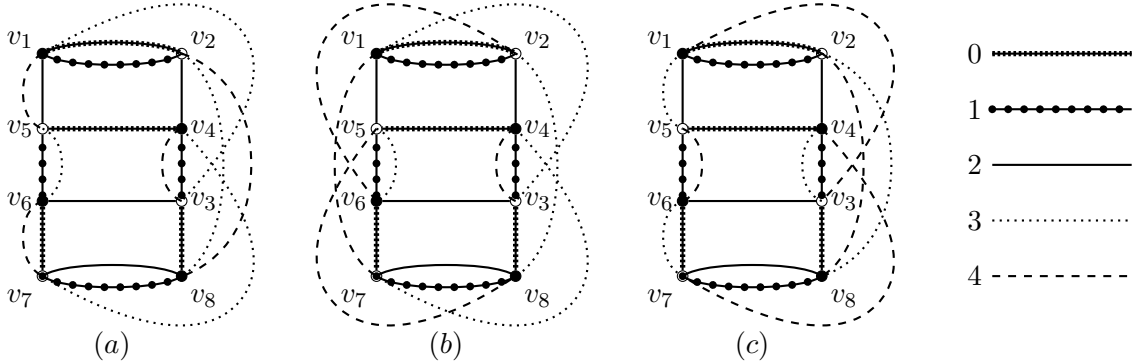
\begin{figure}[ht]
\tikzstyle{vert}=[circle, draw, fill=black!100, inner sep=0pt, minimum width=4pt]
\tikzstyle{vertex}=[circle, draw, fill=black!00, inner sep=0pt, minimum width=4pt]
\tikzstyle{ver}=[]
\tikzstyle{extra}=[circle, draw, fill=black!50, inner sep=0pt, minimum width=2pt]
\tikzstyle{edge} = [draw,thick,-]

\centering
\begin{tikzpicture}[scale=0.37]

\begin{scope}[shift={(-12,7)}]
\node[vertex] (107) at (-2.5,-4){};
\node[vertex] (105) at (-2.5,1.3){};
\node[vertex] (103) at (2.5,-1.3){};
\node[vertex] (102) at (2.5,4){};

\node[vert] (108) at (2.5,-4){};
\node[vert] (104) at (2.5,1.3){};
\node[vert] (106) at (-2.5,-1.3){};
\node[vert] (101) at (-2.5,4){};

\path[edge] (101) -- (105);
\path[edge] (102) -- (104);
\path[edge] (103) -- (106);

\path[edge] (106) -- (105);
\draw [line width=3pt, line cap=round, dash pattern=on 0pt off 2\pgflinewidth] (106) -- (105);
\path[edge] (103) -- (104);
\draw [line width=3pt, line cap=round, dash pattern=on 0pt off 2\pgflinewidth] (103) -- (104);

\path[edge] (108) -- (103);
\draw [line width=2pt, line cap=rectengle, dash pattern=on 1pt off 1] (103) -- (108);
\path[edge] (107) -- (106);
\draw [line width=2pt, line cap=rectengle, dash pattern=on 1pt off 1] (106) -- (107);
\path[edge] (104) -- (105);
\draw [line width=2pt, line cap=rectengle, dash pattern=on 1pt off 1] (104) -- (105);

\draw [line width=2pt, line cap=rectengle, dash pattern=on 1pt off 1] plot [smooth,tension=1.5] coordinates{(101)(0,4.4)(102)};
\draw [line width=3pt, line cap=round, dash pattern=on 0pt off 2\pgflinewidth] plot [smooth,tension=1.5] coordinates{(101)(0,3.6)(102)};
\draw[edge] plot [smooth,tension=1.5] coordinates{(101)(0,4.4)(102)};
\draw[edge] plot [smooth,tension=1.5] coordinates{(101)(0,3.6)(102)};

\draw [line width=3pt, line cap=round, dash pattern=on 0pt off 2\pgflinewidth] plot [smooth,tension=1.5] coordinates{(107)(0,-4.4)(108)};
\draw[edge] plot [smooth,tension=1.5] coordinates{(107)(0,-4.4)(108)};
\draw[edge] plot [smooth,tension=1.5] coordinates{(107)(0,-3.6)(108)};

\node[ver] (200) at (3.3,-1.3){$v_3$};
\node[ver] (201) at (3.3,1.3){$v_4$};
\node[ver] (202) at (3.3,4.5){$v_2$};
\node[ver] (203) at (-3.3,4.5){$v_1$};
\node[ver] (204) at (-3.3,1.3){$v_5$};
\node[ver] (205) at (-3.3,-1.3){$v_6$};
\node[ver] (206) at (-3.3,-4.8){$v_7$};
\node[ver] (207) at (3.3,-4.8){$v_8$};
\draw[edge, dotted] plot [smooth,tension=1.5] coordinates{(101)(4.5,5)(103)};
\draw[edge, dotted] plot [smooth,tension=1.5] coordinates{(102)(4,0)(108)};
\draw[edge, dotted] plot [smooth,tension=1.5] coordinates{(104)(4.5,-5)(107)};
\draw[edge, dotted] plot [smooth,tension=1.5] coordinates{(105)(-1.8,0)(106)};

\draw[edge, dashed] plot [smooth,tension=1.5] coordinates{(101)(-3.2,2.6)(105)};
\draw[edge, dashed] plot [smooth,tension=1.5] coordinates{(102)(5,0)(108)};
\draw[edge, dashed] plot [smooth,tension=1.5] coordinates{(103)(1.8,0)(104)};
\draw[edge, dashed] plot [smooth,tension=1.5] coordinates{(106)(-3.2,-2.6)(107)};
\node[ver] (208) at (0,-6.5){$(a)$};
\end{scope}

\begin{scope}[shift={(0,7)}]
\node[vertex] (107) at (-2.5,-4){};
\node[vertex] (105) at (-2.5,1.3){};
\node[vertex] (103) at (2.5,-1.3){};
\node[vertex] (102) at (2.5,4){};

\node[vert] (108) at (2.5,-4){};
\node[vert] (104) at (2.5,1.3){};
\node[vert] (106) at (-2.5,-1.3){};
\node[vert] (101) at (-2.5,4){};

\path[edge] (101) -- (105);
\path[edge] (102) -- (104);
\path[edge] (103) -- (106);

\path[edge] (106) -- (105);
\draw [line width=3pt, line cap=round, dash pattern=on 0pt off 2\pgflinewidth] (106) -- (105);
\path[edge] (103) -- (104);
\draw [line width=3pt, line cap=round, dash pattern=on 0pt off 2\pgflinewidth] (103) -- (104);

\path[edge] (108) -- (103);
\draw [line width=2pt, line cap=rectengle, dash pattern=on 1pt off 1] (103) -- (108);
\path[edge] (107) -- (106);
\draw [line width=2pt, line cap=rectengle, dash pattern=on 1pt off 1] (106) -- (107);
\path[edge] (104) -- (105);
\draw [line width=2pt, line cap=rectengle, dash pattern=on 1pt off 1] (104) -- (105);

\draw [line width=2pt, line cap=rectengle, dash pattern=on 1pt off 1] plot [smooth,tension=1.5] coordinates{(101)(0,4.4)(102)};
\draw [line width=3pt, line cap=round, dash pattern=on 0pt off 2\pgflinewidth] plot [smooth,tension=1.5] coordinates{(101)(0,3.6)(102)};
\draw[edge] plot [smooth,tension=1.5] coordinates{(101)(0,4.4)(102)};
\draw[edge] plot [smooth,tension=1.5] coordinates{(101)(0,3.6)(102)};

\draw [line width=3pt, line cap=round, dash pattern=on 0pt off 2\pgflinewidth] plot [smooth,tension=1.5] coordinates{(107)(0,-4.4)(108)};
\draw[edge] plot [smooth,tension=1.5] coordinates{(107)(0,-4.4)(108)};
\draw[edge] plot [smooth,tension=1.5] coordinates{(107)(0,-3.6)(108)};

\node[ver] (200) at (3.3,-1.3){$v_3$};
\node[ver] (201) at (3.3,1.3){$v_4$};
\node[ver] (202) at (3.3,4.5){$v_2$};
\node[ver] (203) at (-3.3,4.5){$v_1$};
\node[ver] (204) at (-3.3,1.3){$v_5$};
\node[ver] (205) at (-3.3,-1.3){$v_6$};
\node[ver] (206) at (-3.3,-4.8){$v_7$};
\node[ver] (207) at (3.3,-4.8){$v_8$};
\draw[edge, dotted] plot [smooth,tension=1.5] coordinates{(101)(4.5,5)(103)};
\draw[edge, dotted] plot [smooth,tension=1.5] coordinates{(102)(4,0)(108)};
\draw[edge, dotted] plot [smooth,tension=1.5] coordinates{(104)(4.5,-5)(107)};
\draw[edge, dotted] plot [smooth,tension=1.5] coordinates{(105)(-1.8,0)(106)};

\draw[edge, dashed] plot [smooth,tension=1.5] coordinates{(101)(-4,0)(107)};
\draw[edge, dashed] plot [smooth,tension=1.5] coordinates{(105)(-4,-5)(108)};
\draw[edge, dashed] plot [smooth,tension=1.5] coordinates{(103)(1.8,0)(104)};
\draw[edge, dashed] plot [smooth,tension=1.5] coordinates{(106)(-4,5)(102)};

\node[ver] (208) at (0,-6.5){$(b)$};
\end{scope}

\begin{scope}[shift={(11,7)}]
\node[vertex] (107) at (-2.5,-4){};
\node[vertex] (105) at (-2.5,1.3){};
\node[vertex] (103) at (2.5,-1.3){};
\node[vertex] (102) at (2.5,4){};

\node[vert] (108) at (2.5,-4){};
\node[vert] (104) at (2.5,1.3){};
\node[vert] (106) at (-2.5,-1.3){};
\node[vert] (101) at (-2.5,4){};

\path[edge] (101) -- (105);
\path[edge] (102) -- (104);
\path[edge] (103) -- (106);

\path[edge] (106) -- (105);
\draw [line width=3pt, line cap=round, dash pattern=on 0pt off 2\pgflinewidth] (106) -- (105);
\path[edge] (103) -- (104);
\draw [line width=3pt, line cap=round, dash pattern=on 0pt off 2\pgflinewidth] (103) -- (104);

\path[edge] (108) -- (103);
\draw [line width=2pt, line cap=rectengle, dash pattern=on 1pt off 1] (103) -- (108);
\path[edge] (107) -- (106);
\draw [line width=2pt, line cap=rectengle, dash pattern=on 1pt off 1] (106) -- (107);
\path[edge] (104) -- (105);
\draw [line width=2pt, line cap=rectengle, dash pattern=on 1pt off 1] (104) -- (105);

\draw [line width=2pt, line cap=rectengle, dash pattern=on 1pt off 1] plot [smooth,tension=1.5] coordinates{(101)(0,4.4)(102)};
\draw [line width=3pt, line cap=round, dash pattern=on 0pt off 2\pgflinewidth] plot [smooth,tension=1.5] coordinates{(101)(0,3.6)(102)};
\draw[edge] plot [smooth,tension=1.5] coordinates{(101)(0,4.4)(102)};
\draw[edge] plot [smooth,tension=1.5] coordinates{(101)(0,3.6)(102)};

\draw [line width=3pt, line cap=round, dash pattern=on 0pt off 2\pgflinewidth] plot [smooth,tension=1.5] coordinates{(107)(0,-4.4)(108)};
\draw[edge] plot [smooth,tension=1.5] coordinates{(107)(0,-4.4)(108)};
\draw[edge] plot [smooth,tension=1.5] coordinates{(107)(0,-3.6)(108)};

\node[ver] (200) at (3.3,-1.3){$v_3$};
\node[ver] (201) at (3.3,1.3){$v_4$};
\node[ver] (202) at (3.3,4.5){$v_2$};
\node[ver] (203) at (-3.3,4.5){$v_1$};
\node[ver] (204) at (-3.3,1.3){$v_5$};
\node[ver] (205) at (-3.3,-1.3){$v_6$};
\node[ver] (206) at (-3.3,-4.8){$v_7$};
\node[ver] (207) at (3.3,-4.8){$v_8$};
\draw[edge, dashed] plot [smooth,tension=1.5] coordinates{(101)(4.5,5)(103)};
\draw[edge, dashed] plot [smooth,tension=1.5] coordinates{(102)(4,0)(108)};
\draw[edge, dashed] plot [smooth,tension=1.5] coordinates{(104)(4.5,-5)(107)};
\draw[edge, dashed] plot [smooth,tension=1.5] coordinates{(105)(-1.8,0)(106)};

\draw[edge, dotted] plot [smooth,tension=1.5] coordinates{(101)(-3.2,2.6)(105)};
\draw[edge, dotted] plot [smooth,tension=1.5] coordinates{(102)(5,0)(108)};
\draw[edge, dotted] plot [smooth,tension=1.5] coordinates{(103)(1.8,0)(104)};
\draw[edge, dotted] plot [smooth,tension=1.5] coordinates{(106)(-3.2,-2.6)(107)};
\node[ver] (208) at (0,-6.5){$(c)$};
\end{scope}

\begin{scope}[shift={(22,6)}]
\node[ver] (308) at (-3,5){$0$};
\node[ver] (300) at (-3,3){$1$};
\node[ver] (301) at (-3,1){$2$};
\node[ver] (302) at (-3,-1){$3$};
\node[ver] (303) at (-3,-3){$4$};
\node[ver] (309) at (3,5){};
\node[ver] (304) at (3,3){};
\node[ver] (305) at (3,1){};
\node[ver] (306) at (3,-1){};
\node[ver] (307) at (3,-3){};
\path[edge] (300) -- (304);
\path[edge] (308) -- (309);
\draw [line width=2pt, line cap=rectengle, dash pattern=on 1pt off 1]  (308) -- (309);
\draw [line width=3pt, line cap=round, dash pattern=on 0pt off 2\pgflinewidth]  (300) -- (304);
\path[edge] (301) -- (305);
\path[edge, dotted] (302) -- (306);
\path[edge, dashed] (303) -- (307);
\end{scope}
\end{tikzpicture}
\caption{The three $5$-colored graphs obtained from $\mathcal{G}_2$ and $\mathcal{G}_3$.}
\label{fig:Unique_CP2}
\end{figure}

\medskip
A crystallization of $\mathbb{CP}^2$ can now be found by disjointly adding four more edges to
$\mathcal{G}_1$, $\mathcal{G}_2$, and $\mathcal{G}_3$ to get a $5$-colored bipartite graph. 
A priori, there are $4 \cdot 3 \cdot 2 = 24$ ways to do this per graph but we will see that 
many of them are invalid and the remaining ones are isomorphic to the simple crystallization of
$\mathbb{CP}^2$ presented in Section~\ref{ssec:cp2}.

\begin{description}
	\item[Crystallization $\mathcal{G}_1$:] If $v_5v_i \in \gamma^{-1}(4)$, $i \in \{1,4,6,8\}$, then
		the completed $5$-colored graph $\Gamma$ will contain a $4$-color subgraph $\Gamma_D$ with 
		$\pi_1(|M(\Gamma_{D})|,x) \cong \mathbb{Z}$ and thus $\Gamma$ can not be the crystallization of
		a $4$-manifold.
	\item[Crystallization $\mathcal{G}_2$:] Analogously, we have 
		$v_4v_i \not \in \gamma^{-1}(4)$, $i \in \{2,5,7\}$, since
		otherwise $\Gamma$ can not be the crystallization of a $4$-manifold.
		Thus $v_3v_4\in \gamma^{-1}(4)$ and since 
		$g_{24} = 2$ it follows that we can obtain $\Gamma$ by either adding
		$v_1v_5, v_6v_7,v_2v_8 \in \gamma^{-1}(4)$ (see Figure~\ref{fig:Unique_CP2}~(a)) or adding 
		$v_1v_7, v_2v_6$, $v_5v_8$ $\in \gamma^{-1}(4)$  (see Figure~\ref{fig:Unique_CP2}~(b)). 
		In the former case, $\Gamma$ is isomorphic to the crystallization given in 
		Section~\ref{ssec:cp2}. In the latter case we once again apply 
		Proposition~\ref{prop:gagliardi79b} to see that 
		$\pi_1(|M(\Gamma_{\{0,2,3,4\}})|,x) \cong \mathbb{Z}_2$.
	\item[Crystallization $\mathcal{G}_3$:] Again, we have $v_3v_6$, $v_3v_8$, $v_4v_5$, $v_2v_4 \not \in 
		\gamma^{-1}(4)$ to enforce valid $4$-color subgraphs. It follows that 
		$v_1v_3, v_4v_7 \in \gamma^{-1}(4)$ and since $g_{04}=2$ 
		we have $v_5v_6, v_2v_8 \in \gamma^{-1}(4)$ (see Figure~\ref{fig:Unique_CP2}~(c)). 
		This graph is again isomorphic to the crystallization given in
		Section~\ref{ssec:cp2}.
\end{description}
\end{proof}

\begin{remark}
Given a simple crystallization $(\Gamma,\gamma)$ of a simply connected $4$-manifold $\mathbb{M}$, Theorem~\ref{thm:unique} 
provides a possible way to detect connected summands of type $\mathbb{CP}^2$ or $\overline{\mathbb{CP}^2}$ in $\mathbb{M}$
by solving a subgraph problem: Enumerate all seven vertex subgraphs of type $(\Gamma^1, \gamma^1)$ with one vertex removed 
and check for each if it is connected to the rest of $(\Gamma,\gamma)$ by five edges, one for each color.
Note, that not all connected summands can be detected that way as can be followed from Theorem~\ref{thm:wall} 
(cf. \cite{Wall64SimpConn4Mflds}).
\end{remark}

\section{Heuristics to produce simple crystallizations of 4-manifolds}
\label{sec:compExpts}

Using the $4$-manifold branch of the
computational topology software {\em regina}\footnote{{\em regina} is designed
to handle generalized triangulations in dimensions two, three, and four and can
thus be adapted to work with pseudotriangulations} 
\cite{Burton12CompTopWRegina,Burton09Regina} we use a simulated annealing
type heuristic simplification strategy 
to turn combinatorial manifolds into simple contracted pseudotriangulations.
The strategy uses bistellar moves and so-called
edge contractions (which respect the PL-homeomorphism type of the triangulation,
see Proposition~\ref{prop:edgeContr} and \cite{Burton144Mflds,Pachner87KonstrMethKombHomeo}).

In more detail, recall that a bistellar $i$-move in a simplicial complex $C$
takes $i+1$ facets joined around a common $(d-i)$-dimensional face $\delta \in C$ and 
replaces them with $d-i+1$ facets joined around an $i$-face $\gamma$ with the precondition that
$\gamma$ is not a face of $C$. More precisely we have
$$ \Phi_i (C,\delta) = (C \setminus (\delta \star \partial \gamma)) \cup (\partial \delta \star \gamma),$$
see Figure~\ref{fig:moves} for all bistellar moves in dimension four.
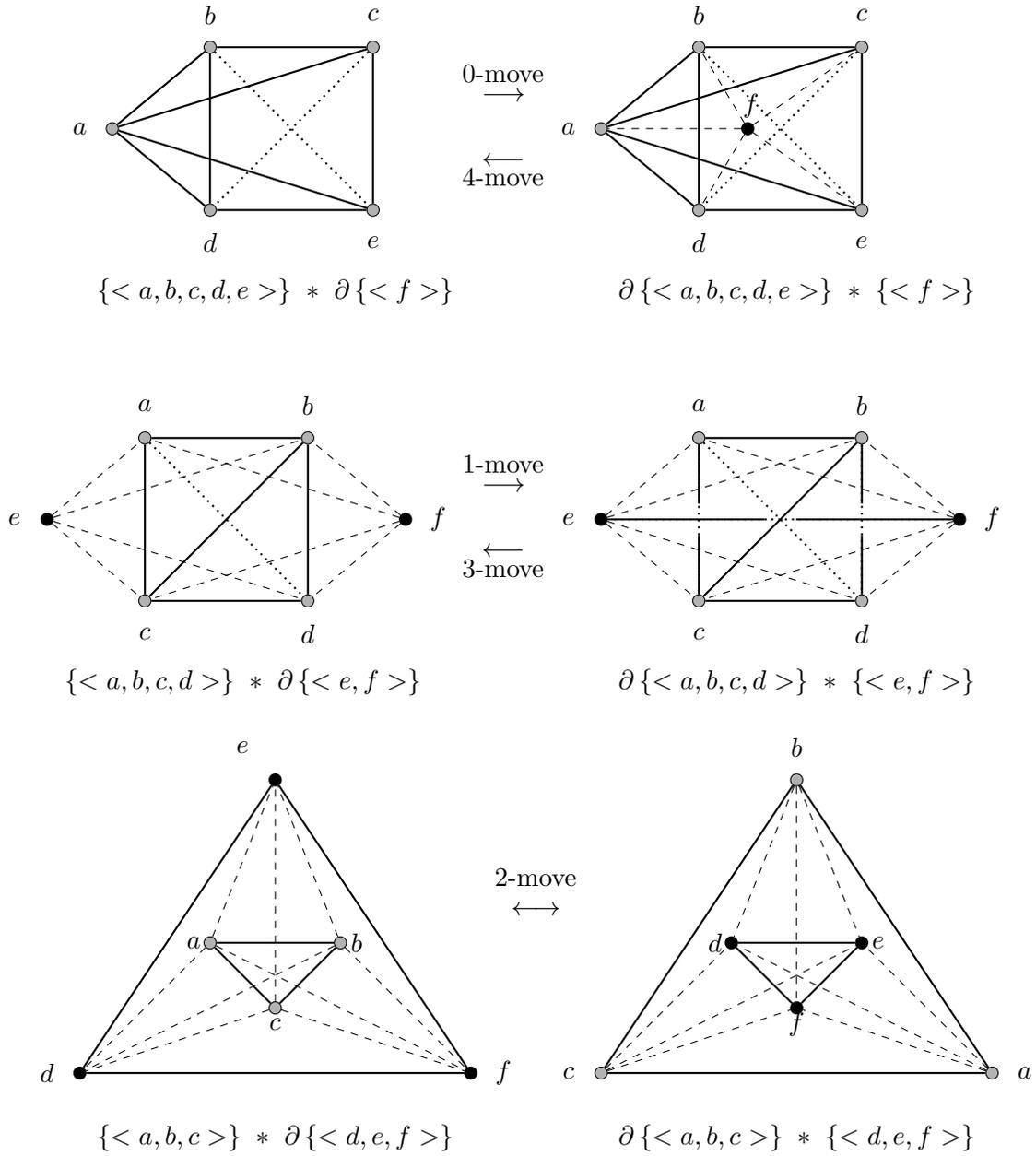
\begin{figure}[p]
\tikzstyle{ver}=[]
\tikzstyle{vertex}=[circle, draw, fill=black!30, inner sep=0pt, minimum width=5pt]
\tikzstyle{vert}=[circle, draw, fill=black!100, inner sep=0pt, minimum width=5pt]
\tikzstyle{edge} = [draw,thick,-]
\tikzstyle{thinedge} = [draw,thin,black!100,-]

\centering
\begin{tikzpicture}[scale=0.47]

\begin{scope}[shift={(0,0)}]
\node[vertex] (a1) at (-13,0){};
\node[vertex] (d1) at (-10,-2.5){}; 
\node[vertex] (b1) at (-10,2.5){}; 
\node[vertex] (e1) at (-5,-2.5){};
\node[vertex] (c1) at (-5,2.5){}; 

\node[ver] () at (-14,0){$a$};
\node[ver] () at (-10,-3.5){$d$}; 
\node[ver] () at (-10,3.5){$b$}; 
\node[ver] () at (-5,-3.5){$e$}; 
\node[ver] () at (-5,3.5){$c$}; 
\foreach \x/\y in {b1/c1,b1/d1,c1/e1,d1/e1}{ \path[edge] (\x) -- (\y);} 
\foreach \x/\y in {b1/e1,c1/d1}{ \path[edge,dotted] (\x) -- (\y);} 
\foreach \x/\y in {a1/b1,a1/c1,a1/d1,a1/e1}{ \path[edge] (\x) -- (\y);}

\node[vertex] (a2) at (2,0){};
\node[vertex] (d2) at (5,-2.5){}; 
\node[vertex] (b2) at (5,2.5){}; 
\node[vertex] (e2) at (10,-2.5){};
\node[vertex] (c2) at (10,2.5){}; 
\node[vert] (f2) at (6.5,0){}; 

\foreach \x/\y in {b2/c2,b2/d2,c2/e2,d2/e2}{ \path[edge] (\x) -- (\y);} 
\foreach \x/\y in {b2/e2,c2/d2}{ \path[edge,dotted] (\x) -- (\y);} 
\foreach \x/\y in {a2/b2,a2/c2,a2/d2,a2/e2}{ \path[edge] (\x) -- (\y);}
\foreach \x/\y in {f2/a2,f2/b2,f2/c2,f2/d2,f2/e2}{ \path[thinedge,dashed] (\x) -- (\y);}

\node[ver] () at (1,0){$a$};
\node[ver] () at (5,-3.5){$d$}; 
\node[ver] () at (5,3.5){$b$}; 
\node[ver] () at (10,-3.5){$e$}; 
\node[ver] () at (10,3.5){$c$}; 
\node[ver] () at (6.6,0.7){$f$};
\node[ver] () at (-1,1){$\longrightarrow $};
\node[ver] () at (-1,-1){$\longleftarrow $};
\node[ver] () at (-1,1.7){$0$-move};
\node[ver] () at (-1,-1.5){$4$-move};
\node[ver] () at (-8,-5){$\{<a,b,c,d,e>\} \,\,\ast\,\, \partial \, \{ < f > \}$};
\node[ver] () at (8,-5){$\partial \, \{<a,b,c,d,e>\}\,\,\ast\,\, \{<f>\}$};
\end{scope}

\begin{scope}[shift={(0,-12)}]
\node[vert] (e3) at (-15,0){};
\node[vertex] (c3) at (-12,-2.5){}; 
\node[vertex] (a3) at (-12,2.5){}; 
\node[vertex] (d3) at (-7,-2.5){};
\node[vertex] (b3) at (-7,2.5){}; 
\node[vert] (f3) at (-4,0){};
\foreach \x/\y in {a3/b3,a3/c3,b3/c3,b3/d3,c3/d3}{ \path[edge] (\x) -- (\y);} 
\foreach \x/\y in {a3/d3}{ \path[edge,dotted] (\x) -- (\y);} 

\node[ver] () at (-16,0){$e$};
\node[ver] () at (-12,-3.5){$c$}; 
\node[ver] () at (-12,3.5){$a$}; 
\node[ver] () at (-7,-3.5){$d$}; 
\node[ver] () at (-7,3.5){$b$}; 
\node[ver] () at (-3,0){$f$}; 

\foreach \x/\y in {e3/a3,e3/b3,e3/c3,e3/d3,f3/a3,f3/b3,f3/c3,f3/d3}{ \path[thinedge, dashed] (\x) -- (\y);}

\node[vert] (e4) at (2,0){};
\node[vertex] (c4) at (5,-2.5){}; 
\node[vertex] (a4) at (5,2.5){}; 
\node[vertex] (d4) at (10,-2.5){};
\node[vertex] (b4) at (10,2.5){}; 
\node[vert] (f4) at (13,0){}; 
\node[ver] (ac1) at (5,0.2){}; 
\node[ver] (ac2) at (5,-0.2){}; 
\node[ver] (bd1) at (10,0.2){}; 
\node[ver] (bd2) at (10,-0.2){};
\node[ver] (ef1) at (7.3,0){}; 
\node[ver] (ef2) at (7.7,0){};

\foreach \x/\y in {a4/b4,a4/ac1,ac2/c4,b4/c4,b4/bd1,bd2/d4,c4/d4,e4/ef1,ef2/f4}{ \path[edge] (\x) -- (\y);} 
\foreach \x/\y in {a4/d4,a4/c4,b4/d4,e4/f4}{ \path[edge,dotted] (\x) -- (\y);} 
\foreach \x/\y in {e4/a4,e4/b4,e4/c4,e4/d4,f4/a4,f4/b4,f4/c4,f4/d4}{ \path[thinedge, dashed] (\x) -- (\y);}

\node[ver] () at (1,0){$e$};
\node[ver] () at (5,-3.5){$c$}; 
\node[ver] () at (5,3.5){$a$}; 
\node[ver] () at (10,-3.5){$d$}; 
\node[ver] () at (10,3.5){$b$}; 
\node[ver] () at (14,0){$f$};

\node[ver] () at (-1,1){$\longrightarrow $};
\node[ver] () at (-1,-1){$\longleftarrow $};
\node[ver] () at (-1,1.7){$1$-move};
\node[ver] () at (-1,-1.5){$3$-move};
\node[ver] () at (-9,-5){$\{<a,b,c,d>\}\,\,\ast\,\, \partial \, \{<e,f>\}$};
\node[ver] () at (8,-5){$\partial \, \{<a,b,c,d>\}\,\,\ast\,\, \{<e,f>\}$};
\end{scope}

\begin{scope}[shift={(0,-24)}]
\node[vert] (d5) at (-14,-5){};
\node[vert] (f5) at (-2,-5){}; 
\node[vert] (e5) at (-8,4){}; 
\node[vertex] (a5) at (-10,-1){};
\node[vertex] (b5) at (-6,-1){}; 
\node[vertex] (c5) at (-8,-3){}; 
\node[ver] () at (-15,-5){$d$};
\node[ver] () at (-1,-5){$f$}; 
\node[ver] () at (-9,5){$e$}; 
\node[ver] () at (-10.5,-1){$a$}; 
\node[ver] () at (-5.5,-1){$b$}; 
\node[ver] () at (-8,-3.5){$c$};

\node[vertex] (a6) at (2,-5){};
\node[vertex] (c6) at (14,-5){}; 
\node[vertex] (b6) at (8,4){}; 
\node[vert] (e6) at (10,-1){};
\node[vert] (d6) at (6,-1){}; 
\node[vert] (f6) at (8,-3){}; 
\node[ver] () at (15,-5){$a$};
\node[ver] () at (1,-5){$c$}; 
\node[ver] () at (8,5){$b$}; 
\node[ver] () at (10.5,-1){$e$}; 
\node[ver] () at (5.5,-1){$d$}; 
\node[ver] () at (8,-3.5){$f$};

\node[ver] () at (0,0){$\longleftrightarrow $};
\node[ver] () at (0,1){$2$-move};

\node[ver] () at (-8,-7){$\{<a,b,c>\}\,\,\ast\,\, \partial \,\{<d,e,f>\}$};
\node[ver] () at (8,-7){$\partial \, \{<a,b,c>\}\,\,\ast\,\,\{<d,e,f>\}$};

\foreach \x/\y in {a5/b5,a5/c5,b5/c5,d5/e5,d5/f5,e5/f5,a6/b6,a6/c6,b6/c6,d6/e6,d6/f6,e6/f6}{ \path[edge] (\x) -- (\y);}
\foreach \x/\y in {a5/d5,a5/e5,a5/f5,c5/d5,c5/e5,c5/f5,b5/d5,b5/e5,b5/f5,a6/d6,a6/e6,a6/f6,c6/d6,c6/e6,c6/f6,b6/d6,b6/e6,b6/f6}{ \path[thinedge,dashed] (\x) -- (\y);}
\end{scope}
\end{tikzpicture}
\caption{Bistellar moves in dimension four.}
\label{fig:moves}
\end{figure}
If $C$ is a pseudotriangulation we can weaken the precondition on $\gamma$ as 
$C$ now can have multiple faces with equal vertex set. In fact, the only precondition
we have to check in the pseudotriangular setting is that no two vertices of any 
facet become identified by a bistellar move.
Note that in dimension four this is automatically satisfied for all $0$-, $2$-, $3$- and $4$-moves.
In the case of a $1$-move we have to check that the edge we are inserting is not a loop (i.e., that
$e \neq f$ in Figure~\ref{fig:moves}).

\medskip
An {\em edge contraction} of a pseudotriangulation $M$ along an edge $e \in M$ 
is the simplicial cell complex $M'$ obtained from $M$ by collapsing every facet 
containing $e$ along $e$ and, in the process, defining the face gluings in
$M'$ in the obvious way. Of course, a number of pre-conditions have to be met
to ensure that $M'$ is again a pseudotriangulation. However, since edge contractions are rare,
we let {\em regina} take care of checking the pre-conditions of the modification in the
more general setting of generalized triangulations
and explicitly check if the complex after an edge contraction is still a pseudotriangulation 
(that is, check that no vertex identifications have been introduced). If the resulting complex 
is not a pseudotriangulation we undo the edge contraction and proceed as before.

It remains to show that edge contractions do not change the PL topological type of a pseudotriangulation
provided the resulting complex is again a pseudotriangulation, more precisely

\begin{prop}\label{prop:edgeContr}
	Let $M$ be a pseudotriangulation of a $4$-manifold, let $e\in M$ be an edge of $M$ and
	let $M'$ be a pseudotriangulation obtained from $M$ by contracting $e$, then
	$ M \, \cong_{\operatorname{PL}} \, M'$.
\end{prop}

\begin{proof}
	Let $u$ and $v$ be the endpoints of $e$ in $M$. Then, the only facets in 
	$M$ containing both $u$ and $v$ are the ones containing $e$. To see this, 
	note that a facet $\Delta \in M$ containing both $u$ and $v$ but not $e$
	would have a loop edge in $M'$ which is a contradiction to the fact that $M'$ is a 
	pseudotriangulation.

	Now, let $\mathbb{D}_u = \operatorname{lk}_M (u) \cap \partial \operatorname{st}_M (e) $ 
	and $\mathbb{D}_v = \operatorname{lk}_M (v) \cap \partial \operatorname{st}_M (e) $
	be the natural decomposition of the boundary of $\operatorname{st}_M (e)$ along
	$\operatorname{lk}_M (e)$.
	Then, by construction $\mathbb{D}_u$ and $\mathbb{D}_v$ are identified in $M'$.
	Denote this subcomplex in $M'$ by $\mathbb{D}$ and without loss of 
	generality identify its center vertex by
	$v \in \mathbb{D}$ (that is, $\mathbb{D}_u$ becomes identified with $\mathbb{D}$ in $M'$).
	The subcomplex $\mathbb{D}$ induces a partition of $\operatorname{st}_{M'} (v)$ into
	$A^+ = \operatorname{st}_{M'} (v) \cap \operatorname{st}_{M'} (e)$ and
	$A^- = \operatorname{st}_{M'} (v) \setminus \operatorname{st}_{M'} (e)$.
	Now, subdivide $A^+$ in $M'$ until it contains a subdivision $\mathbb{D}'$ of a copy
	of $\mathbb{D}$ in its interior such that 
	$\partial \mathbb{D}' = \operatorname{lk}_M (e) = \partial \mathbb{D}$.
	Denote the resulting pseudotriangulation by $M''$. Now we can check that
	$M''$ is a common refinement of both $M$ and $M'$ (with a subdivided version of 
	$\mathbb{D}'$ being identified with a subdivided version of $\mathbb{D}_v$
	in $M''$). Thus, $M$ and $M'$ must be PL-homeomorphic.
\end{proof}

\medskip
Now, the procedure to construct simple contracted pseudotriangulations essentially performs bistellar moves and 
edge contractions at random where moves reducing the complexity of the triangulation ($3$-moves, 
$4$-moves and edge contractions) are performed with a higher
probability. 
Using this strategy we were able to obtain
%
%
$40\,651$ simple contracted pseudotriangulations from the combinatorial manifold PL-homeomorphic to the
K3 surface due to K\"uhnel and the second author \cite{Spreer09CombPropsOfK3} and $19\,129$ simple contracted
pseudotriangulations PL-homeomorphic to the minimum $16$-vertex combinatorial manifold homeomorphic to the K3 surface due to Casella 
and K\"uhnel \cite{Casella01TrigK3MinNumVert}.
We believe that the number of simple contracted pseudotriangulations of the K3 surface
is orders of magnitude larger than the numbers provided above. 
Note that both versions of the $K3$ surface are conjectured to be PL-homeomorphic \cite{Spreer09CombPropsOfK3}. 
This conjecture could be settled by finding a simple crystallization which occurs in both the list of simple contracted
pseudotriangulations. However, as of today both lists are disjoint (see 
\cite{Burton13CombinatorialDiffeomorphisms} for another attempt to
settle this conjecture). This is work in progress.

\medskip
The code, as well as all data generated using the heuristics is available from the authors upon request.

\section{Simple crystallizations of \boldmath{$S^2 \times S^2$}}
\label{eg:s2s2}

Simple crystallizations homeomorphic to $S^2 \times S^2$ have been completely classified by
\cite{Casali12Crystallizations,CasaliDUKEIII,Casali02CodeForMBiPartiteColouredGraphs}, 
and all of them are of standard PL type \cite{Casali12Crystallizations,Casali13ColouredGraphs,
Casali14Cataloguing, Cavicchioli08ClassComb4Mflds}.
Here we present a particularly symmetric example which was obtained from the standard combinatorial
triangulation of $S^2 \times S^2$ using the heuristics described in Section~\ref{sec:compExpts}.

\medskip
Let $(\Gamma^2, \gamma^2)$ be the contracted $5$-colored graph with
color set $C=\{0,1,2,3,4\}$ given in Figure \ref{fig:G2} and let $M_2$ 
be the corresponding contracted pseudotriangulation. To independently verify its topological 
type we can proceed as in the previous example. We follow that
$M_2$ is a simple contracted pseudotriangulation 
of a (simply connected) $4$-manifold with intersection form of rank two.
Taking the barycentric subdivision of $M_2$ yields a combinatorial $4$-manifold
and computing its intersection form using {\em simpcomp} \cite{simpcomp,simpcompISSAC,simpcompISSAC11} yields that
$| M_2 |$ is homeomorphic to $S^2 \times S^2$.

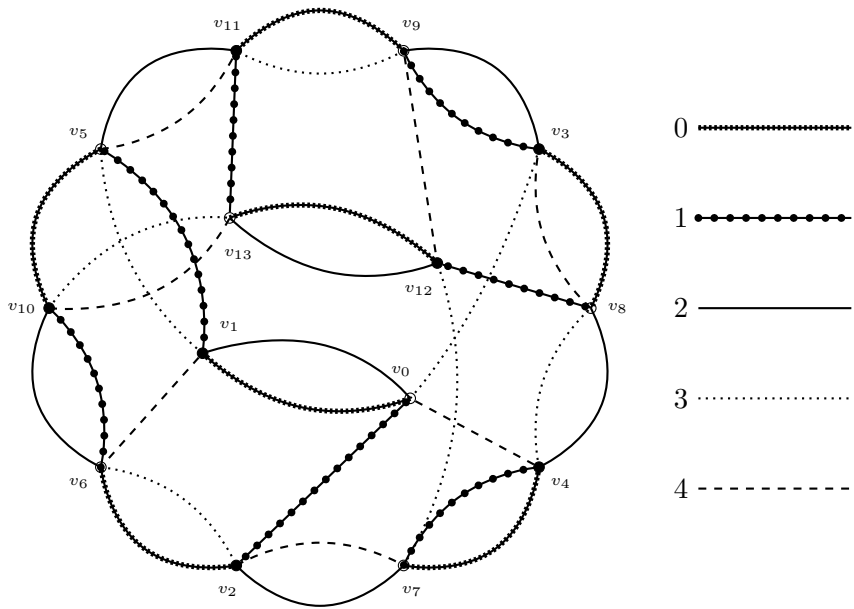
\begin{figure}[ht]
\tikzstyle{vert}=[circle, draw, fill=black!100, inner sep=0pt, minimum width=4pt]
\tikzstyle{vertex}=[circle, draw, fill=black!00, inner sep=0pt, minimum width=4pt]
\tikzstyle{ver}=[]
\tikzstyle{extra}=[circle, draw, fill=black!50, inner sep=0pt, minimum width=2pt]
\tikzstyle{edge} = [draw,thick,-]
\centering
\begin{tikzpicture}[scale=1.2]
\begin{scope}[shift={(0,0)}]
\foreach \x/\y in {0/8,72/9,144/5,216/6,288/7}{
\node[ver] (\y) at (\x:3.3){\tiny{$v_{\y}$}};
\node[vertex] (\y) at (\x:3){};
}
\foreach \x/\y in {36/3,108/11,180/10,252/2,324/4}{
\node[ver] (\y) at (\x:3.3){\tiny{$v_{\y}$}};
\node[vert] (\y) at (\x:3){};}
\end{scope}

\begin{scope}[rotate=18]
\foreach \x/\y in {0/108,72/109,144/105,216/106,288/107}{
\node[ver] (\y) at (\x:3.3){};}
\foreach \x/\y in {36/103,108/1011,180/1010,252/102,324/104}{
\node[ver] (\y) at (\x:3.3){};}
\end{scope}

\begin{scope}[rotate=18]
\foreach \x/\y in {0/208,72/209,144/205,216/206,288/207}{
\node[ver] (\y) at (\x:2.6){};
}
\foreach \x/\y in {36/203,108/2011,180/2010,252/202,324/204}{
\node[ver] (\y) at (\x:2.6){};
}
\end{scope}

\begin{scope}
\node[vertex] (13) at (-1,1){};
\node[vert] (1) at (-1.3,-0.5){};
\node[vertex] (0) at (1,-1){};
\node[vert] (12) at (1.3,0.5){};
\node[ver] () at (-0.9,0.6){\tiny{$v_{13}$}};
\node[ver] () at (-1,-0.2){\tiny{$v_1$}};
\node[ver] () at (0.9,-0.7){\tiny{$v_0$}};
\node[ver] () at (1.1,0.2){\tiny{$v_{12}$}};
\node[ver] (21) at (0.2,1.1){};
\node[ver] (31) at (0.1,0.4){};
\node[ver] (22) at (0,-0.4){};
\node[ver] (32) at (-0.2,-1.1){};
\node[ver] (23) at (-1.5,0.8){};
\node[ver] (33) at (-2.1,0.5){};
\node[ver] (24) at (-2.1,0.8){};
\node[ver] (34) at (-1.8,0.2){};
\node[ver] (25) at (1.8,0.2){};
\node[ver] (35) at (1.5,-1.2){};

\end{scope}

\foreach \x/\y/\z in {8/3/108,9/11/109,5/10/105,6/2/106,7/4/107,4/8/104,3/9/103,11/5/1011,10/6/1010,2/7/102,12/13/21,0/1/22,1/5/23}{ 
\draw[edge] plot [smooth,tension=1] coordinates{(\x)(\z)(\y)};}
\foreach \x/\y/\z in {7/4/207,3/9/203,10/6/2010,0/1/32,12/13/31}{ 
\draw[edge] plot [smooth,tension=1] coordinates{(\x)(\z)(\y)};}
\foreach \x/\y in {0/2,11/13,8/12}{
\draw[edge]  (\x) -- (\y);}

\foreach \x/\y/\z in {0/1/32,2/6/106,3/8/108,4/7/107,9/11/109,5/10/105,12/13/21}{ 
\draw[line width=2pt, line cap=rectengle, dash pattern=on 1pt off 1] plot [smooth,tension=1] coordinates{(\x)(\z)(\y)};}

\foreach \x/\y/\z in {3/9/203,1/5/23,4/7/207,6/10/2010}{
\draw [line width=3pt, line cap=round, dash pattern=on 0pt off 2\pgflinewidth] plot [smooth,tension=1] coordinates{(\x)(\z)(\y)};}

\foreach \x/\y in {0/2,11/13,8/12}{
\draw [line width=3pt, line cap=round, dash pattern=on 0pt off 2\pgflinewidth] (\x) -- (\y);}

\foreach \x/\y/\z in {0/3/25,1/5/33,6/2/206,7/12/35,4/8/204,10/13/24,9/11/209}{ 
\draw[edge,dotted] plot [smooth,tension=1] coordinates{(\x)(\z)(\y)};}

\foreach \x/\y/\z in {8/3/208,2/7/202,5/11/2011,10/13/34}{ 
\draw[edge,dashed] plot [smooth,tension=1] coordinates{(\x)(\z)(\y)};}
\foreach \x/\y in {0/4,1/6,9/12}{
\draw[edge,dashed]  (\x) -- (\y);}
\begin{scope}[shift={(6,2)}]
\node[ver] (300) at (-2,0){$0$};
\node[ver] (301) at (-2,-1){$1$};
\node[ver] (302) at (-2,-2){$2$};
\node[ver] (303) at (-2,-3){$3$};
\node[ver] (304) at (-2,-4){$4$};
\node[ver] (305) at (0,0){};
\node[ver] (306) at (0,-1){};
\node[ver] (307) at (0,-2){};
\node[ver] (308) at (0,-3){};
\node[ver] (309) at (0,-4){};

\path[edge] (300) -- (305);
\draw [line width=2pt, line cap=rectengle, dash pattern=on 1pt off 1]  (300) -- (305);
\path[edge] (301) -- (306);
\draw [line width=3pt, line cap=round, dash pattern=on 0pt off 2\pgflinewidth]  (301) -- (306);
\path[edge] (302) -- (307);
\path[edge, dotted] (303) -- (308);
\path[edge, dashed] (304) -- (309);
\end{scope}
\end{tikzpicture}
\caption{A simple crystallization of $S^2 \times S^2$.}
\label{fig:G2}
\end{figure}

\section{Simple crystallizations of the \boldmath{K3 surface}}
\label{eg:k3}

The following simple crystallization $(\Gamma^3, \gamma^3)$ of the K3 surface was 
obtained from a simple contracted pseudotriangulation $M_3$ which was
constructed using the $17$-vertex combinatorial manifold $K3$ with standard PL structure from 
\cite{Spreer09CombPropsOfK3} together with the heuristics described in Section~\ref{sec:compExpts}.

\medskip
To give an independent prove that $(\Gamma^3, \gamma^3)$ is homeomorphic to the
K3 surface we first have to take a look at the $2$-colored graphs $\Gamma_{\{i,j\}}^3$, 
$\{i,j\} \subset C$. For instance, we have 
$$	\Gamma^3_{\{0,1\}} = 5C_2 \sqcup 7C_4 \sqcup 6C_6 \sqcup C_8 \sqcup 2C_{10}\sqcup 2C_{16}, $$
where $k C_i$ denotes $k$ disjoint copies of an $i$-cycle. Note that $\Gamma^3_{\{0,1\}}$ has exactly $m=23$ 
connected components. More explicitly, using the notation for cycles as given in Section~\ref{prop:connSum} we have

\small
$$ \begin{array}{rc lc lc lc lc lc lc}
	\Gamma^3_{\{0,1\}}&=& C(6,7) &\sqcup& C(88,89) &\sqcup& C(92,93) &\sqcup& C(118,119) &\sqcup& C(120,121) &\sqcup &&\\
	&& \multicolumn{3}{l}{C(12,13,26,27)} &\sqcup& \multicolumn{3}{l}{C(18,19,7,71)} &\sqcup& \multicolumn{3}{l}{C(48,49,108,109)} &\sqcup \\
	&& \multicolumn{3}{l}{C(58,59,82,83)} &\sqcup& \multicolumn{3}{l}{C(66,67,102,103)} &\sqcup& \multicolumn{3}{l}{C(72,73,74,75)} & \sqcup \\
	&& \multicolumn{3}{l}{C(108,109,114,115)} &\sqcup& \multicolumn{3}{l}{C(0,1,20,21,14,15)} &\sqcup& \multicolumn{3}{l}{C(2,3,32,33,8,9)} &\sqcup \\
	&& \multicolumn{5}{l}{C(10,11,84,85,80,81)} &\sqcup& \multicolumn{5}{l}{C(38,39,76,77,40,41)} &\sqcup  \\
	&& \multicolumn{5}{l}{C(54,55,56,57,130,131)} &\sqcup& \multicolumn{5}{l}{C(68,69,86,87,90,91)} & \sqcup \\
	&& \multicolumn{5}{l}{C(62,63,112,113,124,125,64,65)} &\sqcup& \multicolumn{5}{l}{C(16,17,50,51,60,61,22,23,96,97)} & \sqcup \\
	&& \multicolumn{7}{l}{C(34,35,42,43,132,133,128,129,78,79,34,35)} &\sqcup &&&& \\
	&& \multicolumn{7}{l}{C(4,5,52,53,106,107,94,95,36,37,126,127,44,45,46,47)} &\sqcup &&&& \\
	&& \multicolumn{9}{l}{C(24,25,98,99,28,29,110,111,30,31,122,123,100,101,116,117)}.  &&& \\
\end{array} $$

\normalsize
See Figure~\ref{fig:G3} to verify the list of cycles for $\Gamma^3_{\{1,2\}}$ and $\Gamma^3_{\{3,4\}}$. 
In addition, $\Gamma^3_{\{0,1,2\}}$ and $\Gamma^3_{\{0,3,4\}}$, which can be obtained by adding
edges of type $v_{2i}v_{2i+1}$, $0 \leq i \leq 66$, are both connected. 

Furthermore, note that each $(\Gamma^3|_{C \setminus \{c\}}, \gamma^3|_{(\gamma^3)^{-1}(C \setminus \{c\})})$
has $134$ vertices and we have $3 \cdot 23 = 134/2 + 2$. Thus, all
$(\Gamma^3|_{C \setminus \{c\}}, \gamma^3|_{(\gamma^3)^{-1}(C \setminus \{c\})})$ are crystallizations
of $3$-manifolds $\mathbb{M}_3^{(c)}$ for all $c \in C$. It remains to show that for all $c \in C$ the $3$-manifolds $\mathbb{M}_3^{(c)}$ 
are simply connected. Once again, this can be done by applying Proposition \ref{prop:gagliardi79b}.
It follows that $(\Gamma^3,\gamma^3)$ is the simple crystallization of a simply connected $4$-manifold
with intersection form of rank $22$.
At this point we proceed by feeding the corresponding triangulation $M_3$ into {\em simpcomp} to verify that $|M_3|$ is
homeomorphic to the K3 surface.

\begin{figure}[p]
\tikzstyle{vert}=[circle, draw, fill=black!100, inner sep=0pt, minimum width=4pt]
\tikzstyle{vertex}=[circle, draw, fill=black!00, inner sep=0pt, minimum width=4pt]
\tikzstyle{ver}=[]
\tikzstyle{extra}=[circle, draw, fill=black!50, inner sep=0pt, minimum width=2pt]
\tikzstyle{edge} = [draw,thick,-]
\centering

\begin{tikzpicture}[scale=0.51]

\begin{scope}[]
\node[ver] () at (-15,5.5){\tiny{$v_{2i}$}};
\node[ver] (v_{2i+1}) at (-9.5,5.5){\tiny{$v_{2i+1}$}};
\node[vertex] (v_{2i}) at (-14.2,5.5){};
\node[vert] (v_{2i+1}) at (-10.7,5.5){};
\path[edge] (v_{2i}) -- (v_{2i+1});
\draw[line width=2pt, line cap=rectengle, dash pattern=on 1pt off 1] (v_{2i}) -- (v_{2i+1});
\node[ver] () at (-6,5.5){\tiny{for $i \in \{0, \dots, 66\}$}};
\end{scope}

\begin{scope}[shift={(-12,0)}]
\foreach \x/\y in {0/v_{1},60/v_{29},120/v_{93},180/v_{37},240/v_{115},300/v_{71}}{
\node[ver] (\y) at (\x:1.9){\tiny{$\y$}};
\node[vert] (\y) at (\x:2.5){};
}
\foreach \x/\y in {30/v_{20},90/v_{110},150/v_{92},210/v_{126},270/v_{108},330/v_{18}}{
\node[ver] (\y) at (\x:1.9){\tiny{$\y$}};
\node[vertex] (\y) at (\x:2.5){};}

\foreach \x/\y in {v_{1}/v_{20},v_{29}/v_{110},v_{93}/v_{92},v_{37}/v_{126},v_{115}/v_{108},v_{71}/v_{18},v_{20}/v_{29},v_{110}/v_{93},v_{92}/v_{37},v_{126}/v_{115},v_{108}/v_{71},v_{18}/v_{1}}{ 
\path[edge] (\x) -- (\y);}

\foreach \x/\y in {v_{1}/v_{20},v_{29}/v_{110},v_{93}/v_{92},v_{37}/v_{126},v_{115}/v_{108},v_{71}/v_{18}}{
\draw [line width=3pt, line cap=round, dash pattern=on 0pt off 2\pgflinewidth] 	(\x) -- (\y);}
\end{scope}

\begin{scope}[shift={(-6,0)}]
\foreach \x/\y in {0/v_{30},60/v_{90},120/v_{88},180/v_{86},240/v_{128},300/v_{124}}{
\node[ver] (\y) at (\x:1.9){\tiny{$\y$}};
\node[vertex] (\y) at (\x:2.5){};
}
\foreach \x/\y in {30/v_{111},90/v_{87},150/v_{89},210/v_{69},270/v_{133},330/v_{113}}{
\node[ver] (\y) at (\x:1.9){\tiny{$\y$}};
\node[vert] (\y) at (\x:2.5){};}

\foreach \x/\y in {v_{30}/v_{111},v_{90}/v_{87},v_{88}/v_{89},v_{86}/v_{69},v_{128}/v_{133},v_{124}/v_{113},v_{111}/v_{90},v_{87}/v_{88},v_{89}/v_{86},v_{69}/v_{128},v_{133}/v_{124},v_{113}/v_{30}}{ 
\path[edge] (\x) -- (\y);}

\foreach \x/\y in {v_{30}/v_{111},v_{90}/v_{87},v_{88}/v_{89},v_{86}/v_{69},v_{128}/v_{133},v_{124}/v_{113}}{
\draw [line width=3pt, line cap=round, dash pattern=on 0pt off 2\pgflinewidth] 	(\x) -- (\y);}
\end{scope}

\begin{scope}[shift={(0,0)}]
\foreach \x/\y in {0/v_{10},72/v_{78},144/v_{68},216/v_{130},288/v_{76}}{
\node[ver] (\y) at (\x:1.9){\tiny{$\y$}};
\node[vertex] (\y) at (\x:2.5){};
}
\foreach \x/\y in {36/v_{81},108/v_{129},180/v_{91},252/v_{57},324/v_{39}}{
\node[ver] (\y) at (\x:1.9){\tiny{$\y$}};
\node[vert] (\y) at (\x:2.5){};}

\foreach \x/\y in {v_{10}/v_{81},v_{78}/v_{129},v_{68}/v_{91},v_{130}/v_{57},v_{76}/v_{39},v_{81}/v_{78},v_{129}/v_{68},v_{91}/v_{130},v_{57}/v_{76},v_{39}/v_{10}}{ 
\path[edge] (\x) -- (\y);}

\foreach \x/\y in {v_{10}/v_{81},v_{78}/v_{129},v_{68}/v_{91},v_{130}/v_{57},v_{76}/v_{39}}{
\draw [line width=3pt, line cap=round, dash pattern=on 0pt off 2\pgflinewidth] 	(\x) -- (\y);}
\end{scope}

\begin{scope}[shift={(5.5,0)}]
\foreach \x/\y in {0/v_{3},90/v_{119},180/v_{25},270/v_{99}}{
\node[ver] (\y) at (\x:1.6){\tiny{$\y$}};
\node[vert] (\y) at (\x:2.2){};}

\foreach \x/\y in {45/v_{32},135/v_{118},225/v_{98},315/v_{28}}{
\node[ver] (\y) at (\x:1.6){\tiny{$\y$}};
\node[vertex] (\y) at (\x:2.2) {};}

\foreach \x/\y in {v_{32}/v_{3},v_{118}/v_{119},v_{98}/v_{25},v_{28}/v_{99},v_{3}/v_{28},v_{99}/v_{98},v_{25}/v_{118},v_{32}/v_{119}}{
\path[edge] (\x) -- (\y);}
\foreach \x/\y in {v_{32}/v_{3},v_{118}/v_{119},v_{98}/v_{25},v_{28}/v_{99}}{ 
\draw [line width=3pt, line cap=round, dash pattern=on 0pt off 2\pgflinewidth] 	(\x) -- (\y);}
\end{scope}

\begin{scope}[shift={(10.5,0)}]
\foreach \x/\y in {0/v_{24},90/v_{114},180/v_{72},270/v_{70}}{
\node[ver] (\y) at (\x:1.6){\tiny{$\y$}};
\node[vertex] (\y) at (\x:2.2){};}

\foreach \x/\y in {45/v_{117},135/v_{109},225/v_{75},315/v_{19}}{
\node[ver] (\y) at (\x:1.6){\tiny{$\y$}};
\node[vert] (\y) at (\x:2.2) {};}

\foreach \x/\y in {v_{117}/v_{24},v_{109}/v_{114},v_{75}/v_{72},v_{19}/v_{70},v_{24}/v_{19},v_{70}/v_{75},v_{72}/v_{109},v_{117}/v_{114}}{
\path[edge] (\x) -- (\y);}
\foreach \x/\y in {v_{117}/v_{24},v_{109}/v_{114},v_{75}/v_{72},v_{19}/v_{70}}{ 
\draw [line width=3pt, line cap=round, dash pattern=on 0pt off 2\pgflinewidth] 	(\x) -- (\y);}
\end{scope}

\begin{scope}[shift={(-9,-4)}]
\foreach \x/\y in {30/v_{0},150/v_{12},270/v_{2}}{
\node[ver] (\y) at (\x:1){\tiny{$\y$}};
\node[vertex] (\y) at (\x:1.6){};}

\foreach \x/\y in {90/v_{15},210/v_{27},330/v_{9}}{
\node[ver] (\y) at (\x:1){\tiny{$\y$}};
\node[vert] (\y) at (\x:1.6) {};}

\foreach \x/\y in {v_{0}/v_{15},v_{12}/v_{27},v_{2}/v_{9},v_{15}/v_{12},v_{27}/v_{2},v_{9}/v_{0}}{
\path[edge] (\x) -- (\y);}
\foreach \x/\y in {v_{0}/v_{15},v_{12}/v_{27},v_{2}/v_{9}}{
\draw [line width=3pt, line cap=round, dash pattern=on 0pt off 2\pgflinewidth] 	(\x) -- (\y);
}
\end{scope}

\begin{scope}[shift={(-3,-4)}]
\foreach \x/\y in {30/v_{5},150/v_{83},270/v_{17}}{
\node[ver] (\y) at (\x:1){\tiny{$\y$}};
\node[vert] (\y) at (\x:1.6){};}

\foreach \x/\y in {90/v_{52},210/v_{58},330/v_{50}}{
\node[ver] (\y) at (\x:1){\tiny{$\y$}};
\node[vertex] (\y) at (\x:1.6) {};}

\foreach \x/\y in {v_{5}/v_{52},v_{83}/v_{58},v_{17}/v_{50},v_{52}/v_{83},v_{58}/v_{17},v_{50}/v_{5}}{
\path[edge] (\x) -- (\y);}
\foreach \x/\y in {v_{5}/v_{52},v_{83}/v_{58},v_{17}/v_{50}}{
\draw [line width=3pt, line cap=round, dash pattern=on 0pt off 2\pgflinewidth] 	(\x) -- (\y);
}
\end{scope}

\begin{scope}[shift={(3,-4)}]
\foreach \x/\y in {30/v_{16},150/v_{94},270/v_{82}}{
\node[ver] (\y) at (\x:1){\tiny{$\y$}};
\node[vertex] (\y) at (\x:1.6){};}

\foreach \x/\y in {90/v_{97},210/v_{107},330/v_{59}}{
\node[ver] (\y) at (\x:1){\tiny{$\y$}};
\node[vert] (\y) at (\x:1.6) {};}

\foreach \x/\y in {v_{16}/v_{97},v_{94}/v_{107},v_{82}/v_{59},v_{97}/v_{94},v_{107}/v_{82},v_{59}/v_{16}}{
\path[edge] (\x) -- (\y);}
\foreach \x/\y in {v_{16}/v_{97},v_{94}/v_{107},v_{82}/v_{59}}{
\draw [line width=3pt, line cap=round, dash pattern=on 0pt off 2\pgflinewidth] 	(\x) -- (\y);
}
\end{scope}

\begin{scope}[shift={(8,-4)}]
\foreach \x/\y in {30/v_{40},150/v_{56},270/v_{54}}{
\node[ver] (\y) at (\x:1){\tiny{$\y$}};
\node[vertex] (\y) at (\x:1.6){};}

\foreach \x/\y in {90/v_{77},210/v_{55},330/v_{131}}{
\node[ver] (\y) at (\x:1){\tiny{$\y$}};
\node[vert] (\y) at (\x:1.6) {};}

\foreach \x/\y in {v_{40}/v_{77},v_{56}/v_{55},v_{54}/v_{131},v_{77}/v_{56},v_{55}/v_{54},v_{131}/v_{40}}{
\path[edge] (\x) -- (\y);}
\foreach \x/\y in {v_{40}/v_{77},v_{56}/v_{55},v_{54}/v_{131}}{
\draw [line width=3pt, line cap=round, dash pattern=on 0pt off 2\pgflinewidth] 	(\x) -- (\y);
}
\end{scope}

\begin{scope}[shift={(-12,-8)}]
\foreach \x/\y in {0/v_{66},120/v_{48},240/v_{74}}{
\node[ver] (\y) at (\x:1){\tiny{$\y$}};
\node[vertex] (\y) at (\x:1.6){};}

\foreach \x/\y in {60/v_{103},180/v_{105},300/v_{73}}{
\node[ver] (\y) at (\x:1){\tiny{$\y$}};
\node[vert] (\y) at (\x:1.6) {};}

\foreach \x/\y in {v_{66}/v_{103},v_{48}/v_{105},v_{74}/v_{73},v_{103}/v_{48},v_{105}/v_{74},v_{73}/v_{66}}{
\path[edge] (\x) -- (\y);}
\foreach \x/\y in {v_{66}/v_{103},v_{48}/v_{105},v_{74}/v_{73}}{
\draw [line width=3pt, line cap=round, dash pattern=on 0pt off 2\pgflinewidth] 	(\x) -- (\y);
}
\end{scope}

\begin{scope}[]
\foreach \x/\z/\w in {-14/v_{132}/v_{43},-12/v_{8}/v_{33},-6/v_{46}/v_{45},0/v_{106}/v_{53},5.5/v_{112}/v_{63},11/v_{116}/v_{101}}{
\node[ver] () at (\x,-3){\tiny{$\w$}};
\node[vert] (\w) at (\x,-3.4){};
\node[ver] () at (\x,-5.8){\tiny{$\z$}};
\node[vertex] (\z) at (\x,-5.4){};}

\foreach \x/\y/\z/\w in {v_{132}/v_{43}/-14.5/-13.5,v_{8}/v_{33}/-12.5/-11.5,v_{46}/v_{45}/-6.5/-5.5,v_{106}/v_{53}/-0.5/0.5,v_{112}/v_{63}/5/6,v_{116}/v_{101}/10.5/11.5}{
\draw[edge] plot [smooth,tension=1] coordinates{(\x)(\z,-4.4)(\y)};
\draw[edge] plot [smooth,tension=1] coordinates{(\x)(\w,-4.4)(\y)};
\draw[line width=3pt, line cap=round, dash pattern=on 0pt off 2\pgflinewidth] plot [smooth,tension=1] coordinates{(\x)(\z,-4.3)(\y)};}
\end{scope}

\begin{scope}[]
\foreach \x/\z/\w in {-9/v_{13}/v_{26},-5.5/v_{65}/v_{62},-2/v_{35}/v_{42},1.5/v_{85}/v_{80},5/v_{67}/v_{102},8.5/v_{31}/v_{122}}{
\node[ver] () at (\x,-7){\tiny{$\w$}};
\node[vertex] (\w) at (\x,-7.4){};
\node[ver] () at (\x,-9.8){\tiny{$\z$}};
\node[vert] (\z) at (\x,-9.4){};}

\foreach \x/\z/\w in {-6.5/v_{14}/v_{21},-3/v_{34}/v_{79},0.5/v_{64}/v_{125},4/v_{100}/v_{123},7.5/v_{104}/v_{49},11/v_{120}/v_{121}}{
\node[ver] () at (\x,-7){\tiny{$\w$}};
\node[vert] (\w) at (\x,-7.4){};
\node[ver] () at (\x,-9.8){\tiny{$\z$}};
\node[vertex] (\z) at (\x,-9.4){};}
\foreach \x/\y in {v_{13}/v_{26},v_{26}/v_{21},v_{21}/v_{14},v_{14}/v_{13},v_{34}/v_{79},v_{79}/v_{62},v_{62}/v_{65},v_{65}/v_{34},v_{42}/v_{35},v_{35}/v_{64},v_{64}/v_{125},v_{125}/v_{42},v_{80}/v_{85},v_{85}/v_{100},v_{100}/v_{123},v_{123}/v_{80},v_{67}/v_{102},v_{102}/v_{49},v_{49}/v_{104},v_{104}/v_{67},v_{31}/v_{122},v_{122}/v_{121},v_{121}/v_{120},v_{120}/v_{31}}{
\path[edge] (\x) -- (\y);}
\foreach \x/\y in {v_{13}/v_{26},v_{21}/v_{14},v_{34}/v_{79},v_{62}/v_{65},v_{42}/v_{35},v_{64}/v_{125},v_{80}/v_{85},v_{100}/v_{123},v_{67}/v_{102},v_{49}/v_{104},v_{31}/v_{122},v_{121}/v_{120}}{
\draw [line width=3pt, line cap=round, dash pattern=on 0pt off 2\pgflinewidth] 	(\x) -- (\y);
}
\end{scope}

\begin{scope}[]
\foreach \x/\z/\w in {-15/v_{4}/v_{7},-8/v_{44}/v_{51},-1/v_{84}/v_{61},6/v_{38}/v_{23},13.2/v_{36}/v_{95}}{
\node[ver] () at (\x,4){\tiny{$\z$}};
\node[vertex] (\w) at (\x,3.6){};
\node[ver] () at (\x,-11.4){\tiny{$\w$}};
\node[vert] (\z) at (\x,-11){};}

\foreach \x/\z/\w in {-11.5/v_{47}/v_{6},-4.5/v_{127}/v_{60},2.5/v_{11}/v_{22},9.5/v_{41}/v_{96}}{
\node[ver] () at (\x,4){\tiny{$\z$}};
\node[vert] (\w) at (\x,3.6){};
\node[ver] () at (\x,-11.4){\tiny{$\w$}};
\node[vertex] (\z) at (\x,-11){};}
\foreach \x/\y in {v_{4}/v_{47},v_{44}/v_{127},v_{84}/v_{11},v_{38}/v_{41},v_{36}/v_{95},v_{96}/v_{23},v_{22}/v_{61},v_{60}/v_{51},v_{6}/v_{7},
v_{47}/v_{44},v_{127}/v_{84},v_{11}/v_{38},v_{41}/v_{36},v_{95}/v_{96},v_{23}/v_{22},v_{61}/v_{60},v_{51}/v_{6},v_{7}/v_{4}}{
\path[edge] (\x) -- (\y);}

\foreach \x/\y in {v_{4}/v_{47},v_{44}/v_{127},v_{84}/v_{11},v_{38}/v_{41},v_{36}/v_{95},v_{96}/v_{23},v_{22}/v_{61},v_{60}/v_{51},v_{6}/v_{7}}{
\draw [line width=3pt, line cap=round, dash pattern=on 0pt off 2\pgflinewidth] 	(\x) -- (\y);}
\end{scope}
\begin{scope}[shift={(0,-8)}]
\foreach \x/\z/\w in {-14/v_{0}/v_{9},-11/v_{22}/v_{61},-8/v_{32}/v_{121},-5/v_{36}/v_{79},-2/v_{40}/v_{109},1/v_{50}/v_{103},4/v_{64}/v_{63},7/v_{80}/v_{123},10/v_{84}/v_{59}}{
\node[ver] () at (\x,-7){\tiny{$\w$}};
\node[vert] (\w) at (\x,-7.4){};
\node[ver] () at (\x,-9.8){\tiny{$\z$}};
\node[vertex] (\z) at (\x,-9.4){};}

\foreach \x/\z/\w in {-12/v_{7}/v_{2},-9/v_{27}/v_{24},-6/v_{67}/v_{122},-3/v_{95}/v_{76},0/v_{5}/v_{48},3/v_{131}/v_{130},6/v_{29}/v_{86},9/v_{57}/v_{116},12/v_{87}/v_{98}}{
\node[ver] () at (\x,-7){\tiny{$\w$}};
\node[vertex] (\w) at (\x,-7.4){};
\node[ver] () at (\x,-9.8){\tiny{$\z$}};
\node[vert] (\z) at (\x,-9.4){};}

\foreach \x/\y in {v_{9}/v_{2},v_{7}/v_{0},v_{61}/v_{24},v_{27}/v_{22},v_{121}/v_{122},v_{67}/v_{32},v_{79}/v_{76},v_{95}/v_{36},v_{109}/v_{48},
v_{5}/v_{40},v_{103}/v_{130},v_{131}/v_{50},v_{64}/v_{29},v_{86}/v_{63},v_{123}/v_{116},v_{57}/v_{80},v_{84}/v_{87},v_{98}/v_{59}}{
\path[edge, dashed] (\x) -- (\y);}
\foreach \x/\y in {v_{0}/v_{9},v_{2}/v_{7},v_{22}/v_{61},v_{24}/v_{27},v_{32}/v_{121},v_{122}/v_{67},v_{36}/v_{79},v_{76}/v_{95},v_{40}/v_{109},
v_{48}/v_{5},v_{50}/v_{103},v_{130}/v_{131},v_{63}/v_{64},v_{29}/v_{86},v_{80}/v_{123},v_{116}/v_{57},v_{59}/v_{84},v_{87}/v_{98}}{
\path[edge, dotted] 	(\x) -- (\y);}

\foreach \x/\z/\w in {-15/v_{1}/v_{16},-7.5/v_{65}/v_{92},0/v_{85}/v_{124},8/v_{51}/v_{112}}{
\node[ver] () at (\x,-5.8){\tiny{$\z$}};
\node[vert] (\w) at (\x,-6.2){};
\node[ver] () at (\x,-11.4){\tiny{$\w$}};
\node[vertex] (\z) at (\x,-11){};}

\foreach \x/\z/\w in {-11.5/v_{8}/v_{97},-4/v_{62}/v_{69},4/v_{58}/v_{113},13.2/v_{60}/v_{83}}{
\node[ver] () at (\x,-5.8){\tiny{$\z$}};
\node[vertex] (\w) at (\x,-6.2){};
\node[ver] () at (\x,-11.4){\tiny{$\w$}};
\node[vert] (\z) at (\x,-11){};}

\foreach \x/\y in {v_{8}/v_{65},v_{62}/v_{85},v_{58}/v_{51},v_{60}/v_{83},v_{112}/v_{113},v_{124}/v_{69},v_{92}/v_{97},v_{16}/v_{1}}{
\path[edge, dashed] (\x) -- (\y);}
\foreach \x/\y in {v_{1}/v_{8},v_{65}/v_{62},v_{85}/v_{58},v_{51}/v_{60},v_{83}/v_{112},v_{113}/v_{124},v_{69}/v_{92},v_{97}/v_{16}}{
\path[edge, dotted] 	(\x) -- (\y);}

\foreach \x/\z/\w in {-15/v_{14}/v_{11},-7.5/v_{46}/v_{25},0/v_{18}/v_{107},8/v_{52}/v_{71}}{
\node[ver] () at (\x,-11.8){\tiny{$\z$}};
\node[vertex] (\w) at (\x,-12.2){};
\node[ver] () at (\x,-22.4){\tiny{$\w$}};
\node[vert] (\z) at (\x,-22){};}

\foreach \x/\z/\w in {-11.5/v_{89}/v_{82},-4/v_{49}/v_{118},4/v_{19}/v_{106},13.2/v_{35}/v_{38}}{
\node[ver] () at (\x,-11.8){\tiny{$\z$}};
\node[vert] (\w) at (\x,-12.2){};
\node[ver] () at (\x,-22.4){\tiny{$\w$}};
\node[vertex] (\z) at (\x,-22){};}

\foreach \x/\y in {v_{89}/v_{46},v_{49}/v_{18},v_{19}/v_{52},v_{35}/v_{38},v_{71}/v_{106},v_{107}/v_{118},v_{25}/v_{82},v_{11}/v_{14}}{
\path[edge, dashed] (\x) -- (\y);}
\foreach \x/\y in {v_{14}/v_{89},v_{46}/v_{49},v_{18}/v_{19},v_{52}/v_{35},v_{38}/v_{71},v_{106}/v_{107},v_{118}/v_{25},v_{82}/v_{11}}{
\path[edge, dotted]	(\x) -- (\y);}
\end{scope}

\begin{scope}[shift={(-12,-23)}]
\foreach \x/\y in {0/v_{4},90/v_{42},180/v_{70},270/v_{34}}{
\node[ver] (\y) at (\x:1.6){\tiny{$\y$}};
\node[vertex] (\y) at (\x:2.2){};}

\foreach \x/\y in {45/v_{47},135/v_{43},225/v_{39},315/v_{41}}{
\node[ver] (\y) at (\x:1.6){\tiny{$\y$}};
\node[vert] (\y) at (\x:2.2) {};}

\foreach \x/\y in {v_{4}/v_{41},v_{34}/v_{39},v_{70}/v_{43},v_{47}/v_{42}}{
\path[edge, dashed] (\x) -- (\y);}
\foreach \x/\y in {v_{47}/v_{4},v_{43}/v_{42},v_{39}/v_{70},v_{41}/v_{34}}{ 
\path[edge, dotted] 	(\x) -- (\y);}
\end{scope}

\begin{scope}[shift={(-6.1,-23)}]
\foreach \x/\y in {0/v_{10},90/v_{68},180/v_{72},270/v_{12}}{
\node[ver] (\y) at (\x:1.6){\tiny{$\y$}};
\node[vertex] (\y) at (\x:2.2){};}

\foreach \x/\y in {45/v_{75},135/v_{91},225/v_{73},315/v_{15}}{
\node[ver] (\y) at (\x:1.6){\tiny{$\y$}};
\node[vert] (\y) at (\x:2.2) {};}

\foreach \x/\y in {v_{10}/v_{15},v_{12}/v_{73},v_{72}/v_{91},v_{75}/v_{68}}{
\path[edge, dashed] (\x) -- (\y);}
\foreach \x/\y in {v_{75}/v_{10},v_{91}/v_{68},v_{73}/v_{72},v_{15}/v_{12}}{ 
\path[edge, dotted] 	(\x) -- (\y);}
\end{scope}

\begin{scope}[shift={(-0.2,-23)}]
\foreach \x/\y in {0/v_{20},90/v_{90},180/v_{110},270/v_{26}}{
\node[ver] (\y) at (\x:1.6){\tiny{$\y$}};
\node[vertex] (\y) at (\x:2.2){};}

\foreach \x/\y in {45/v_{21},135/v_{111},225/v_{93},315/v_{13}}{
\node[ver] (\y) at (\x:1.6){\tiny{$\y$}};
\node[vert] (\y) at (\x:2.2) {};}

\foreach \x/\y in {v_{20}/v_{13},v_{26}/v_{93},v_{110}/v_{111},v_{21}/v_{90}}{
\path[edge, dashed] (\x) -- (\y);}
\foreach \x/\y in {v_{21}/v_{20},v_{111}/v_{90},v_{93}/v_{110},v_{13}/v_{26}}{ 
\path[edge, dotted] 	(\x) -- (\y);}
\end{scope}

\begin{scope}[shift={(5.2,-23)}]
\foreach \x/\y in {0/v_{33},90/v_{125},180/v_{117},270/v_{81}}{
\node[ver] (\y) at (\x:1.6){\tiny{$\y$}};
\node[vert] (\y) at (\x:2.2){};}

\foreach \x/\y in {45/v_{44},135/v_{74},225/v_{56},315/v_{66}}{
\node[ver] (\y) at (\x:1.6){\tiny{$\y$}};
\node[vertex] (\y) at (\x:2.2) {};}

\foreach \x/\y in {v_{33}/v_{66},v_{81}/v_{56},v_{117}/v_{74},v_{44}/v_{125}}{
\path[edge, dashed] (\x) -- (\y);}
\foreach \x/\y in {v_{44}/v_{33},v_{74}/v_{125},v_{56}/v_{117},v_{66}/v_{81}}{ 
\path[edge, dotted] 	(\x) -- (\y);}
\end{scope}

\begin{scope}[shift={(10.5,-23)}]
\foreach \x/\y in {0/v_{94},90/v_{104},180/v_{126},270/v_{128}}{
\node[ver] (\y) at (\x:1.6){\tiny{$\y$}};
\node[vertex] (\y) at (\x:2.2){};}

\foreach \x/\y in {45/v_{115},135/v_{119},225/v_{133},315/v_{37}}{
\node[ver] (\y) at (\x:1.6){\tiny{$\y$}};
\node[vert] (\y) at (\x:2.2) {};}

\foreach \x/\y in {v_{94}/v_{37},v_{128}/v_{133},v_{126}/v_{119},v_{115}/v_{104}}{
\path[edge, dashed] (\x) -- (\y);}
\foreach \x/\y in {v_{115}/v_{94},v_{119}/v_{104},v_{133}/v_{126},v_{37}/v_{128}}{ 
\path[edge, dotted] 	(\x) -- (\y);}
\end{scope}

\begin{scope}[shift={(-9,-26.5)}]
\foreach \x/\y in {30/v_{28},150/v_{102},270/v_{30}}{
\node[ver] (\y) at (\x:1){\tiny{$\y$}};
\node[vertex] (\y) at (\x:1.6){};}

\foreach \x/\y in {90/v_{99},210/v_{31},330/v_{3}}{
\node[ver] (\y) at (\x:1){\tiny{$\y$}};
\node[vert] (\y) at (\x:1.6) {};}

\foreach \x/\y in {v_{99}/v_{102},v_{31}/v_{30},v_{3}/v_{28}}{
\path[edge, dashed] (\x) -- (\y);}
\foreach \x/\y in {v_{28}/v_{99},v_{102}/v_{31},v_{30}/v_{3}}{
\path[edge, dotted] (\x) -- (\y);
}
\end{scope}

\begin{scope}[shift={(-3,-26.5)}]
\foreach \x/\y in {30/v_{78},150/v_{132},270/v_{114}}{
\node[ver] (\y) at (\x:1){\tiny{$\y$}};
\node[vertex] (\y) at (\x:1.6){};}

\foreach \x/\y in {90/v_{129},210/v_{105},330/v_{77}}{
\node[ver] (\y) at (\x:1){\tiny{$\y$}};
\node[vert] (\y) at (\x:1.6) {};}

\foreach \x/\y in {v_{129}/v_{132},v_{105}/v_{114},v_{77}/v_{78}}{
\path[edge, dashed] (\x) -- (\y);}
\foreach \x/\y in {v_{78}/v_{129},v_{132}/v_{105},v_{114}/v_{77}}{
\path[edge, dotted] (\x) -- (\y);
}
\end{scope}

\begin{scope}[shift={(3,-26.5)}]
\foreach \x/\y in {30/v_{23},150/v_{17},270/v_{101}}{
\node[ver] (\y) at (\x:1){\tiny{$\y$}};
\node[vert] (\y) at (\x:1.6){};}

\foreach \x/\y in {90/v_{54},210/v_{100},330/v_{96}}{
\node[ver] (\y) at (\x:1){\tiny{$\y$}};
\node[vertex] (\y) at (\x:1.6) {};}

\foreach \x/\y in {v_{54}/v_{17},v_{100}/v_{101},v_{96}/v_{23}}{
\path[edge, dashed] (\x) -- (\y);}
\foreach \x/\y in {v_{23}/v_{54},v_{17}/v_{100},v_{101}/v_{96}}{
\path[edge, dotted] (\x) -- (\y);
}
\end{scope}

\begin{scope}[shift={(0,2)}]
\foreach \x/\z/\w in {-12/v_{6}/v_{55},-6.1/v_{88}/v_{127},-0.2/v_{108}/v_{53},7.2/v_{120}/v_{45}}{
\node[ver] () at (\x,-28){\tiny{$\w$}};
\node[vert] (\w) at (\x,-28.5){};
\node[ver] () at (\x,-30.9){\tiny{$\z$}};
\node[vertex] (\z) at (\x,-30.5){};}

\foreach \x/\y/\z/\w in {v_{6}/v_{55}/-12.5/-11.5,v_{88}/v_{127}/-6.6/-5.6,v_{108}/v_{53}/-0.7/0.3,v_{120}/v_{45}/6.7/7.7}{
\draw[edge, dotted] plot [smooth,tension=1] coordinates{(\x)(\w,-29.5)(\y)};
\draw[edge, dashed] plot [smooth,tension=1] coordinates{(\x)(\z,-29.5)(\y)};}

\foreach \x/\y/\z/\w in {-15/-12/0/5,-11/-8/1/6,-7/-4/2/7,-3/0/3/8,1/4/4/9}{
\node[ver] (\z) at (\x,-34){\tiny{$\z$}};
\node[ver] (\w) at (\y,-34){\tiny{}};}
\path[edge] (0) -- (5);
\path[edge] (1) -- (6);
\path[edge] (2) -- (7);
\path[edge, dotted] (3) -- (8);
\path[edge, dashed] (4) -- (9);
\draw[line width=2pt, line cap=rectengle, dash pattern=on 1pt off 1] (0) -- (5);
\draw [line width=3pt, line cap=round, dash pattern=on 0pt off 2\pgflinewidth] 	(1) -- (6);

\node[ver] () at (0,-36){$\Gamma^3= \Gamma^3_{\{0\}} \cup \Gamma^3_{\{1,2\}} \cup \Gamma^3_{\{3,4\}}$};
\end{scope}

\end{tikzpicture}
\caption{A simple crystallization of the K3 surface. Note that both $\Gamma^3_{\{1,2\}}$ and $\Gamma^3_{\{3,4\}}$
	have $23$ connected components and both become connected if we add the $0$-colored edges, that is,
	if we connect $v_{2i}$ to $v_{2i+1}$ for $0 \leq i \leq 66$.\label{fig:G3} }
\end{figure}

\section{Simple crystallizations of pairs of homeomorphic but not PL-homeomorphic 4-manifolds}

In Sections \ref{eg:cp2}, \ref{eg:s2s2}, and  \ref{eg:k3} we constructed simple crystallizations of  $\mathbb{CP}^2$,
$S^2 \times S^2$ and the K3 surface respectively. Now using Lemma~\ref{Lemma:connected-sum}, we can construct simple 
crystallizations of simply connected $4$-manifolds of type $k \,\mathbb{CP}^2 \,\, \#  \,\, l \,\overline{\mathbb{CP}^2}$ 
and $m \, K3 \,\, \# \,\, r \, S^2 \times S^2$, for all $k,l,m,r \geq 0$. As of today, these are all known topological types
of simply connected $4$-manifolds which allow at least one PL structure.

The following result about connected sums of simply connected $4$-manifolds is due to Wall.

\begin{theo}[Wall \cite{Wall64SimpConn4Mflds}]
	\label{thm:wall}
	Let $\mathbb{M}$ and $\mathbb{N}$ be two simply connected closed PL $4$-manifolds with isomorphic
	intersection forms. Then there exist a $k \geq 0$ such that
	$\mathbb{M} \# k(S^2 \times S^2)$ and $\mathbb{N} \# k(S^2 \times S^2)$ are PL-homeomorphic.
\end{theo}

Furthermore, it is known that $k$ in the above theorem is not always equal to zero.

\begin{theo}[Kronheimer and Mrowka \cite{Kronheimer94RecRels4MfldInvs}]
	\label{thm:exotic}
	$$K3 \,\,\#\,\, \overline{\mathbb{CP}^2} \qquad \not \cong_{PL} \qquad 3 \,\mathbb{CP}^2 \,\,\#\,\, 20 \, \overline{\mathbb{CP}^2}.$$
\end{theo}

Note that $K3 \,\#\, \overline{\mathbb{CP}^2}$ and $3 \,\mathbb{CP}^2 \,\#\, 20 \, \overline{\mathbb{CP}^2}$
have isomorphic odd intersection forms and are thus homeomorphic by Theorem~\ref{Freedman1}.

Theorem~\ref{thm:exotic} was proven by computing the Donaldson polynomial \cite{Donaldson90DonPoly} for both
manifolds. The Donaldson polynomial is a powerful PL-homeomorphism invariant
to tell homeomorphic but not PL-homeomorphic $4$-manifolds apart.
However, it is usually very hard to compute for a pair of given manifolds.
Now, using the connected-sum property of simple contracted pseudotriangulations (cf. Lemma~\ref{Lemma:connected-sum})
it follows that

\begin{cor}
	\label{cor:nondiff}
	There is a pair of simple contracted pseudotriangulations of homeomorphic
	but non-PL-homeomorphic simply connected $4$-manifolds.
\end{cor}

Since simple contracted pseudotriangulations can be regarded as (strongly) minimal
pseudotriangulations this also addresses a number of problems posed in \cite{Spreer09CombPropsOfK3}
in a pseudotriangular setting.

\begin{remark}
Proposition~\ref{prop:iforms} together with Freedman's
classification theorem (cf. Theorem~\ref{Freedman1}) gives us many
homeomorphic pairs of simply connected $4$-manifolds with
distinct connected sum decompositions: For each simply connected
PL $4$-manifold with even intersection form build the connected
sum with $\mathbb{CP}^2$ or $\overline{\mathbb{CP}^2}$ and
compare it to the suitable connected sum of the form 
$k \mathbb{CP}^2 \,\#\, l \overline{\mathbb{CP}^2}$. Hence, we can assume that more pairs of
simply connected PL $4$-manifolds exist which (i) can be build from our
simple crystallizations of $\mathbb{CP}^2$, $S^2 \times S^2 $, and
$K3$ (i.e., which are simply connected PL $4$-manifolds of ``standard type'') 
and which  (ii) require $k > 0$ in Theorem~\ref{thm:wall}.
\end{remark}

\medskip
The code for producing these pairs of homeomorphic but not PL-homeomorphic connected $4$-manifolds 
is available from the authors upon request.

%

\section*{Acknowledgements}

The authors express their gratitude to Prof. Basudeb Datta for helpful comments.
In particular, the Definition~\ref{defn:simple} and Lemma~\ref{Lemma:simple-b} 
are due to him. Furthermore, the authors want to thank the anonymous referees
for insightful and helpful remarks about this work. 

The first author 
is supported by CSIR, India for SPM Fellowship and the UGC Centre for Advanced 
Studies. The second author is supported by DIICCSRTE, Australia and DST, 
India, under the Australia-India Strategic Research Fund (project AISRF06660).

{\footnotesize
 \bibliographystyle{abbrv}
 \bibliography{/home/jonathan/bibliography/bibliography}

\begin{thebibliography}{10}

\bibitem{Basak13MinCryst}
B.~Basak and B.~Datta.
\newblock Minimal crystallizations of 3-manifolds.
\newblock {\em Electron. J. Combin.}, 21(1):\# P 1.61, 1--25, 2014.

\bibitem{Bjoerner00SimplMnfBistellarFlips}
A.~Bj{\"o}rner and F.~H. Lutz.
\newblock {S}implicial manifolds, bistellar flips and a 16-vertex triangulation
  of the {P}oincar{\'e} homology 3-sphere.
\newblock {\em Experiment. Math.}, 9(2):275--289, 2000.

\bibitem{Bondy08GraphTheory}
J.~A. Bondy and U.~S.~R. Murty.
\newblock {\em Graph theory}, volume 244 of {\em Graduate Texts in
  Mathematics}.
\newblock Springer, New York, 2008.

\bibitem{Burton11Census}
B.~A. Burton.
\newblock Detecting genus in vertex links for the fast enumeration of
  3-manifold triangulations.
\newblock In {\em {ISSAC} 2011: Proceedings of the 36th International Symposium
  on Symbolic and Algebraic Computation}, pages 59--66. ACM, 2011.

\bibitem{Burton11PachnerGraph}
B.~A. Burton.
\newblock The {P}achner graph and the simplification of 3-sphere
  triangulations.
\newblock In {\em Computational geometry ({SCG}'11)}, pages 153--162. ACM, New
  York, 2011.

\bibitem{Burton12CompTopWRegina}
B.~A. Burton.
\newblock Computational topology with regina: Algorithms, heuristics and
  implementations.
\newblock In {\em Geometry \& Topology Down Under}, pages 195--224. American
  Mathematical Society, 2012.

\bibitem{Burton09Regina}
B.~A. Burton, R.~Budney, W.~Pettersson, et~al.
\newblock Regina: normal surface and 3-manifold topology software, version
  4.95.
\newblock {\tt http://\allowbreak regina.\allowbreak sourceforge.\allowbreak
  net/}, 1999--2013.

\bibitem{Burton13CombinatorialDiffeomorphisms}
B.~A. Burton and J.~Spreer.
\newblock Computationally proving triangulated $4$-manifolds to be
  diffeomorphic.
\newblock \texttt{arXiv:1403.2780 [math.GT]}, 2013.
\newblock 29th ACM Symposium on Computational Geometry, Young Researchers
  Forum, Collections of abstracts, 2013, pages 15--16.

\bibitem{Burton144Mflds}
B.~A. Burton and J.~Spreer.
\newblock Combinatorial diffeomorphisms for triangulated $4$-manifolds, 2014.
\newblock In preparation.

\bibitem{Casali12Crystallizations}
M.~R. Casali.
\newblock Catalogues of {PL}-manifolds and complexity estimations via
  crystallization theory.
\newblock In {\em Triangulations}, volume~24 of {\em Oberwolfach Report}, pages
  1469--1471. {EMS Publishing House}, 2012.

\bibitem{CasaliDUKEIII}
M.~R. Casali and P.~Cristofori.
\newblock {DUKE III: A program to handle edge-coloured graphs representing PL
  $n$-dimensional manifolds}.
\newblock
  \texttt{http://cdm.unimo.it/home/matematica/casali.mariarita/DUKEIII.htm},
  2012.

\bibitem{Casali13ColouredGraphs}
M.~R. Casali and P.~Cristofori.
\newblock Coloured graphs representing pl 4-manifolds.
\newblock {\em Electronic Notes in Discrete Mathematics}, pages 83--87, 2013.

\bibitem{Casali14Cataloguing}
M.~R. Casali and P.~Cristofori.
\newblock Cataloguing pl 4-manifolds by gem-complexity.
\newblock \texttt{arXiv:1408.0378v1 [math.GT]}, 2014.
\newblock Preprint, 23 pages, 5 figures.

\bibitem{Casali14GemCompl}
M.~R. Casali, P.~Cristofori, and C.~Gagliardi.
\newblock A characterization of {PL} $4$-manifolds admitting simple
  crystallizations.
\newblock \texttt{arXiv:1410.3321v1 [math.GT]}, 2014.
\newblock Preprint, 12 pages.

\bibitem{Casali02CodeForMBiPartiteColouredGraphs}
M.~R. Casali and C.~Gagliardi.
\newblock A code for {$m$}-bipartite edge-coloured graphs.
\newblock {\em Rend. Istit. Mat. Univ. Trieste}, 32(suppl. 1):55--76 (2002),
  2001.
\newblock Dedicated to the memory of Marco Reni.

\bibitem{Casella01TrigK3MinNumVert}
M.~Casella and W.~K{\"u}hnel.
\newblock {A} triangulated {$K3$} surface with the minimum number of vertices.
\newblock {\em Topology}, 40(4):753--772, 2001.

\bibitem{Cavicchioli80NormDecClsdNMflds}
A.~Cavicchioli, L.~Grasselli, and M.~Pezzana.
\newblock A normal decomposition for closed {$n$}-manifolds.
\newblock {\em Boll. Un. Mat. Ital. B (5)}, 17(3):1146--1165, 1980.

\bibitem{Cavicchioli08ClassComb4Mflds}
A.~Cavicchioli and F.~Spaggiari.
\newblock Classifying combinatorial 4-manifolds up to complexity.
\newblock {\em Bol. Soc. Mat. Mexicana (3)}, 14(2):303--319, 2008.

\bibitem{Chiavacci93LinkingMinTrig}
R.~Chiavacci, P.~Cristofori, and C.~Gagliardi.
\newblock Linking two minimal triangulations of {${\bf C}{\rm P}^2$}.
\newblock In {\em Proceedings of the {E}leventh {I}nternational {C}onference of
  {T}opology ({T}rieste, 1993)}, volume~25, pages 127--140 (1994), 1993.

\bibitem{Donaldson83GaugeTheory4Mflds}
S.~K. Donaldson.
\newblock {A}n application of gauge theory to four-dimensional topology.
\newblock {\em J. Differential Geom.}, 18(2):279--315, 1983.

\bibitem{Donaldson90DonPoly}
S.~K. Donaldson.
\newblock Polynomial invariants for smooth four-manifolds.
\newblock {\em Topology}, 29(3):257--315, 1990.

\bibitem{simpcomp}
F.~Effenberger and J.~Spreer.
\newblock simpcomp - a {GAP} package, {V}ersion 2.0.0.
\newblock \url{https://code.google.com/p/simpcomp}, 2009--2014.

\bibitem{simpcompISSAC}
F.~Effenberger and J.~Spreer.
\newblock simpcomp - a {GAP} toolbox for simplicial complexes.
\newblock {\em ACM Communications in Computer Algebra}, 44(4):186 -- 189, 2010.

\bibitem{simpcompISSAC11}
F.~Effenberger and J.~Spreer.
\newblock Simplicial blowups and discrete normal surfaces in the {GAP} package
  simpcomp.
\newblock {\em ACM Communications in Computer Algebra}, 45(3):173 -- 176, 2011.

\bibitem{Ferri82S2xS2}
M.~Ferri and C.~Gagliardi.
\newblock On the genus of {$4$}-dimensional products of manifolds.
\newblock {\em Geom. Dedicata}, 13(3):331--345, 1982.

\bibitem{Ferri86GraphTheoryCrystallizations}
M.~Ferri, C.~Gagliardi, and L.~Grasselli.
\newblock A graph-theoretical representation of {PL}-manifolds---a survey on
  crystallizations.
\newblock {\em Aequationes Math.}, 31(2-3):121--141, 1986.

\bibitem{Freedman82Top4DimMnf}
M.~Freedman.
\newblock {T}he topology of four-dimensional manifolds.
\newblock {\em J. Differential Geom.}, 17:357--453, 1982.

\bibitem{Freedman90TopOf4Mflds}
M.~H. Freedman and F.~Quinn.
\newblock {\em {T}opology of 4-manifolds}, volume~39 of {\em Princeton
  Mathematical Series}.
\newblock Princeton University Press, Princeton, NJ, 1990.

\bibitem{Furuta01MonopoleEq}
M.~Furuta.
\newblock Monopole equation and the {$\frac{11}8$}-conjecture.
\newblock {\em Math. Res. Lett.}, 8(3):279--291, 2001.

\bibitem{Gagliardo79CombCritCrystallizations}
C.~Gagliardi.
\newblock A combinatorial characterization of {$3$}-manifold crystallizations.
\newblock {\em Boll. Un. Mat. Ital. A (5)}, 16(3):441--449, 1979.

\bibitem{Gagliardi79FundGrpClsdNMfld}
C.~Gagliardi.
\newblock How to deduce the fundamental group of a closed {$n$}-manifold from a
  contracted triangulation.
\newblock {\em J. Combin. Inform. System Sci.}, 4(3):237--252, 1979.

\bibitem{Gagliardo89CP2}
C.~Gagliardi.
\newblock On the genus of the complex projective plane.
\newblock {\em Aequationes Math.}, 37(2-3):130--140, 1989.

\bibitem{Gompf}
R.~E. Gompf and A.~I. Stipsicz.
\newblock {\em {$4$}-manifolds and {K}irby calculus}, volume~20 of {\em
  Graduate Studies in Mathematics}.
\newblock American Mathematical Society, Providence, RI, 1999.

\bibitem{Hatcher2002AlgTop}
A.~Hatcher.
\newblock {\em Algebraic Topology}.
\newblock Cambridge University Press, 2002.

\bibitem{Jaco03ZeroEffTriang}
W.~Jaco and J.~H. Rubinstein.
\newblock {$0$}-efficient triangulations of 3-manifolds.
\newblock {\em J. Differential Geom.}, 65(1):61--168, 2003.

\bibitem{Kronheimer94RecRels4MfldInvs}
P.~B. Kronheimer and T.~S. Mrowka.
\newblock Recurrence relations and asymptotics for four-manifold invariants.
\newblock {\em Bull. Amer. Math. Soc. (N.S.)}, 30(2):215--221, 1994.

\bibitem{Kuehnel83The9VertComplProjPlane}
W.~K{\"u}hnel and T.~F. Banchoff.
\newblock {T}he {$9$}-vertex complex projective plane.
\newblock {\em Math. Intelligencer}, 5(3):11--22, 1983.

\bibitem{Kuehnel83Uniq3Nb4MnfFewVert}
W.~K{\"u}hnel and G.~Lassmann.
\newblock {T}he unique {$3$}-neighborly {$4$}-manifold with few vertices.
\newblock {\em J. Combin. Theory Ser. A}, 35(2):173--184, 1983.

\bibitem{Lutz03TrigMnfFewVertVertTrans}
F.~H. Lutz.
\newblock {\em {T}riangulated manifolds with few vertices and vertex-transitive
  group actions}.
\newblock PhD thesis, TU Berlin, Aachen, 1999.

\bibitem{Lutz11TrigMnflds}
F.~H. Lutz.
\newblock Triangulating manifolds.
\newblock Springer, in press, ISBN 978-3-540-34502-2, 2014.

\bibitem{Matsumoto82ElevenEight}
Y.~Matsumoto.
\newblock On the bounding genus of homology {$3$}-spheres.
\newblock {\em J. Fac. Sci. Univ. Tokyo Sect. IA Math.}, 29(2):287--318, 1982.

\bibitem{McDuff95SymplecticTop}
D.~McDuff and D.~Salamon.
\newblock {\em Introduction to symplectic topology}.
\newblock Oxford Mathematical Monographs. The Clarendon Press, Oxford
  University Press, New York, 1995.
\newblock Oxford Science Publications.

\bibitem{Milnor73SymmBilForms}
J.~Milnor and D.~Husemoller.
\newblock {\em {S}ymmetric bilinear forms}, volume~73 of {\em {E}rgebnisse der
  {M}athematik und ihrer {G}renzgebiete}.
\newblock Springer-Verlag, New York, 1973.

\bibitem{Perelman07PC}
J.~Morgan and G.~Tian.
\newblock {\em Ricci flow and the {P}oincar\'e conjecture}, volume~3 of {\em
  Clay Mathematics Monographs}.
\newblock American Mathematical Society, Providence, RI, 2007.

\bibitem{Pachner87KonstrMethKombHomeo}
U.~Pachner.
\newblock {K}onstruktionsmethoden und das kombinatorische
  {H}o\-m\"{o}o\-mor\-phie\-pro\-blem f\"{u}r {T}ri\-an\-gul\-ier\-ung\-en
  kompakter semilinearer {M}annigfaltigkeiten.
\newblock {\em Abh. Math. Sem. Uni. Hamburg}, 57:69--86, 1987.

\bibitem{Pezzana74Crystallizations}
M.~Pezzana.
\newblock Sulla struttura topologica delle variet\`a compatte.
\newblock {\em Ati Sem. Mat. Fis. Univ. Modena}, 23(1):269--277 (1975), 1974.

\bibitem{Rohlin84NewResults4Mflds}
V.~A. Rohlin.
\newblock {N}ew results in the theory of four-dimensional manifolds.
\newblock {\em Doklady Akad. Nauk SSSR (N.S.)}, 84:221--224, 1952.

\bibitem{Scorpan05WildWorldOf4Mflds}
A.~Scorpan.
\newblock {\em The wild world of 4-manifolds}.
\newblock American Mathematical Society, Providence, RI, 2005.

\bibitem{Spreer10Diss}
J.~Spreer.
\newblock {\em Blowups, slicings and permutation groups in combinatorial
  topology}.
\newblock PhD thesis, University of Stuttgart, 2011.
\newblock Ph.D. thesis.

\bibitem{Spreer09CombPropsOfK3}
J.~Spreer and W.~K{\"u}hnel.
\newblock {C}ombinatorial properties of the {K}3 surface: {S}implicial blowups
  and slicings.
\newblock {\em Experiment. Math.}, 20(2):201--216, 2011.

\bibitem{Wall64SimpConn4Mflds}
C.~T.~C. Wall.
\newblock On simply-connected {$4$}-manifolds.
\newblock {\em J. London Math. Soc.}, 39:141--149, 1964.

\bibitem{Weil58K3Surf}
A.~Weil.
\newblock {\em {\OE}uvres scientifiques. {C}ollected papers. {V}olume {II}
  (1951--1964)}.
\newblock Springer-Verlag, Berlin, 2009.
\newblock Reprint of the 1979 original.

\end{thebibliography}
}

\end{document}